\documentclass[10pt]{article}

\usepackage{amssymb}
\usepackage{amsmath}
\usepackage{amsthm}
\usepackage{bbm}
\usepackage{mathrsfs} 
\usepackage{esint} 
\usepackage{graphicx}
\usepackage{xcolor}
\usepackage{hyperref}
\usepackage[numbers]{natbib}

\usepackage{bookmark}

\hypersetup{pdfstartview={FitH}}

\usepackage[utf8]{inputenc}

\usepackage[labelfont=bf,labelsep=space]{caption}  

\newtheorem{theorem}{Theorem}[section]
\newtheorem{prop}[theorem]{Proposition}
\newtheorem{lem}[theorem]{Lemma}

\newtheorem{defn}[theorem]{Definition}

\newcommand{\R}{\mathbbm{R}}

\newcommand{\N}{\mathbbm{N}}

\renewcommand{\leq}{\leqslant}
\renewcommand{\geq}{\geqslant}

\newcommand{\comment}[1]{}
\numberwithin{equation}{section}

\title{Joint upper Banach density, VC dimensions and Euclidean point configurations}
\author{Bruno Predojević\\
Department of Mathematics, Faculty of Science, University of Zagreb\\ Bijeni\v{c}ka cesta 30, 10000 Zagreb, Croatia\\
\texttt{bruno.predojevic@math.hr}}

\date{\vspace{-5ex}}

\begin{document}
\maketitle 
\begin{abstract}
We study two related quantities which generalize the concept of upper Banach density of a set to two measurable subsets of the plane. The first of them allows us to generalize a classic result on sufficiently large distances realized in a set of positive upper density, to distances between points of two sets satisfying an appropriate density condition. The second one allows us to show that for all sufficiently large scales $t>0$ and for a smooth, closed, centrally symmetric, planar curve $\Gamma$ which bounds a convex and compact region in the plane and is of non-vanishing curvature, the family consisting of portions of translates of $t\Gamma$ has the maximal possible Vapnik--Chervonenkis dimension.
 \end{abstract}
 
\begin{center}
\end{center}
\section{Introduction}
\subsection{A brief overview of previous results}
In areas as diverse as geometric measure theory, ergodic theory and combinatorics, one often strives to capture the intuitive notion of the proportion of space occupied by a given subset. One of the ways to do this rigorously is by considering the so-called upper Banach density, which is defined for a measurable set $A \subseteq \mathbbm{R}^d$ as
\begin{equation}\label{def: upper Banach density}
\overline{\delta}(A):=\limsup_{R \to \infty} \sup_{z \in \mathbbm{R}^d} \frac{\vert A \cap (z+[0,R]^d)\vert}{R^d}.    
\end{equation}
{In the appendix, we show that the limit superior in the above definition can be replaced by an actual limit.} Szek\'ely conjectured that all planar sets with positive upper Banach density necessarily realize all sufficiently large distances between pairs of their points. The conjecture was further popularized by Erd\H{o}s in \cite{Erd83:open}, sparking developments across a broad range of fields. Eventually, three distinct proofs were found: one by Furstenberg, Katznelson and Weiss \cite{FKW90:dist}; one by Falconer and Marstrand \cite{FM86:dist}; and one by Bourgain \cite{B86:roth}. {In fact, Bourgain's proof solved a generalized version of the conjecture. Namely, it holds that all measurable sets $A \subseteq \mathbbm{R}^{d+1}$ of positive upper Banach density necessarily contain the vertices of all sufficiently large dilates of an isometric copy of any non-degenerate $d$-dimensional simplex.}
This result sparked a whole series of investigations, determining which finite configurations have the property that isometric copies of all their sufficiently large dilations can be found inside sets of {positive upper Banach density}. Graham in \cite{Gra94:countex} gave the currently most general negative answer known. In particular, configurations which cannot be inscribed inside a unit sphere in $\mathbbm{R}^{d}$, cannot satisfy the analogue of Bourgain's simplex theorem. 
The most general positive result is due to Lyall and Magyar \cite{LM16:prod} and it concerns configurations that are Cartesian products of simplices. 

A fruitful path to further generalizations opens up if one is satisfied with dropping the condition of finding isometric copies of the studied configuration and instead considers ``flexible configurations'', namely distance graphs, which can be thought of as graphs which carry information about the lengths of their edges. The study of Bourgain-type results for distance graphs was pioneered by Lyall and Magyar in \cite{LM20}. {In \cite{K20:anisotrop}, among results for rigid configurations, namely simplices and boxes, Kova\v{c} obtained results for anisotropic scaling of a special class of distance graphs, namely distance trees. In \cite{KP2023} Kova\v{c} and the present author obtained dimensionally sharp embedding results for yet another special class of distance graphs, namely hypercube graphs.
Another possible generalization concerns finding configurations in very dense sets which were studied in \cite{FKY22} and \cite{KSS2026}.}
In this paper, we propose a new direction for generalization of these configuration results. Instead of having a single ambient set in which we are to look for given patterns, we study two fixed ambient sets and configurations which in some prescribed way have some of their vertices {lying }in one and the remaining ones in the other set.
\subsection{Statement of the first main result}
The simplest configuration one might consider finding between two sets is the configuration consisting of only two points, i.e., we ask under which conditions on measurable subsets $A, B \subseteq \mathbbm{R}^2$ can we say that all sufficiently large distances between them are realized. One might think that asking that both sets have  positive upper Banach density might be enough to guarantee that an analogue of the aforementioned result (answering Szek\'ely's question) holds, however it is rather straightforward to construct a counterexample. We let $A$ be the union of all balls centered around $(2^n,0)$ of radius $n$ and we let $B$ be the union of balls centered at $(-2^n,0)$ of radius $n$, where $n$ ranges over the positive integers. {Numerous other naive attempts of defining a certain ``averaged'' upper density of $A$ and $B$ also fail. For instance, one may take a union of annuli $10^n<\lvert x\rvert<2 \cdot 10^n$ and the second set to be the union of the annuli $6 \cdot 10^n<\lvert x\rvert<7 \cdot 10^n$ when $n$ ranges over the positive integers.} It turns out that it is crucial to have control over ``how close'' the points can get to each other.
\begin{defn}\label{def: Joint density for 2 sets}
    For measurable sets $A,B \subseteq \mathbbm{R}^2$, we define their joint upper Banach density as
    \[
    \overline{\delta}(A,B) := \sup_{M \geq 1} \limsup_{R \to +\infty} \sup_{z \in \mathbbm{R}^2} \inf_{M \leq r \leq \frac{R}{2}}
    \frac{1}{(2r)^2R^2}\int_{z+[0,R]^2}\int_{\mathbbm{R}^2} \mathbbm{1}_A(x)\mathbbm{1}_B(y)\mathbbm{1}_{[-r,r]^2}(x-y) \,
    \textup{d}y\,\textup{d}x.
    \]
\end{defn}
 A consequence of Proposition \ref{prop: New condition comparability} below tells us that, for a measurable set $A \subseteq \mathbbm{R}^2$, it follows that 
\[
\overline{\delta}(A,A)>0 \quad \Longleftrightarrow \quad \overline{\delta}(A) >0,
\]
thus the qualitative condition of a set having positive upper Banach density generalizes to a set having positive joint upper Banach density, in the sense of Definition \ref{def: Joint density for 2 sets}. Thus Theorem \ref{thm: two set Szekely qualitative} below is indeed a generalization of the previously mentioned classic result.
\begin{theorem}[Two-set Szek\'ely's problem]\label{thm: two set Szekely qualitative}
   Let $A,B \subseteq \mathbbm{R}^2$ be measurable sets such that $\overline{\delta}(A,B) > 0$. Then, there exists a large enough $\lambda_0=\lambda_0(A,B) > 0$ such that for all $\lambda \geq \lambda_0$, there exist $x \in A$ and $y \in B$ such that
    \[
    \lvert x- y \rvert = \lambda.
    \] 
\end{theorem}

A prototypical example of sets $A$ and $B$ for which $\overline{\delta}(A,B)>0$ are the squares of a chess-board pattern in the plane. 
\subsection{VC dimension and statement of second main result}
The concept of VC dimension was first introduced by Vapnik and Chervonenkis in \cite{VC1971} and it is used to formalize the idea of complexity of a family of sets. Here, we briefly restate it. For an excellent brief explanation of the connections with statistical learning theory, we advise the reader to look at the introduction of \cite{FIMW21}.
\begin{defn}
    Let $X$ be some ambient set and let $\mathcal{T} \subseteq \mathcal{P}(X)$ be some family of subsets of $X$. We say that a finite set $C$ of cardinality $n\in \N$ {is shattered by} $\mathcal{T}$ if for every $D \subseteq C$, there exists a $T \in \mathcal{T}$ such that 
    \[
    T\cap C = D.
    \]
\end{defn}
\begin{defn}
 Let $X$ be some ambient set and let $\mathcal{T} \subseteq \mathcal{P}(X)$ be some family of subsets of $X$. We say that $\mathcal{T}$ has Vapnik–-Chervonenkis (VC) dimension equal to $n$ if $n$ is the largest positive integer for which there exists a finite set $C \subseteq X$ of cardinality $n$ which is shattered by $\mathcal{T}$; we denote this by
\[
VC(\mathcal{T})=n.
\]   
 If no such positive integer $n$ exists, then we set
\[
VC(\mathcal{T})=0.
\]
 If arbitrarily large such positive integers $n$ exist, then we set
\[
VC(\mathcal{T})=\infty.
\]
\end{defn}
 Euclidean geometry provides us with the following elementary example of a family for which we can explicitly determine the VC dimension. Let $E \subseteq \R^d$ and $t>0$ be arbitrary. Consider the family of $E$-centered spheres of radius $t$ intersected with $E$,
    \[
    \mathcal{T}_{t}(E):=\{ S_t(e) \cap E : e \in E\},
    \]
    where
    \[
    S_t(e) := \{y \in \R^d : \lvert e - y \rvert =t\}.
    \]
From elementary geometry, chiefly the fact that $d+1$ distinct points uniquely determine a $(d-1)$-sphere in $d$ dimensions, one can easily deduce that
\[
VC(\mathcal{T}_{t}(\mathbbm{R}^d))=d+1
\]
for every $t>0$. The interested reader might also consult Definition 1.5.5. and the comments thereafter in \cite{McD2023}.

In \cite{FIMW21}, Fitzpatrick, Iosevich, McDonald and Wyman made the connection between VC dimension of families of circles and Ramsey theory-type questions, namely they asked how large, where ``large'' should be understood in a suitable sense, must the set $E$ be, before it becomes unavoidable for the family $\mathcal{T}_t(E)$ to have the maximal possible VC dimension. They showed that in the setting of $2$-dimensional vector spaces over finite fields $\mathbbm{F}_q^2$, with the appropriate modifications of the above definitions, if one has $\lvert E \rvert\geq C q^{\frac{15}{8}}$, then 
\[
VC(\mathcal{T}_{t}(E))=3,
\]
which can be seen to be the maximal dimension possible in their setting. An analogous version of the above claim in higher dimensional vector spaces over finite fields was studied \cite{ABCIJLXMcDMMRV2024}. In fact, Ramsey theory-type questions relating to VC dimensions have, in recent years, sparked a whole series of investigations, such as \cite{ABCIJLMcDMMRV2025}, \cite{IJMcD2021}, \cite{IMcDS2023}, \cite{McDSW2025}. For further results and a nice overview of problems in the context of finite fields, we advise the reader to look at McDonald's doctoral dissertation \cite{McD2023}, where the reader can also find results in the Euclidean setting, with the notion of largeness measured by the Hausdorff dimension. In fact, measuring largeness via Hausdorff dimension seems to be a very fruitful approach in the Euclidean setting; the interested reader can consult \cite{IMMcDMcD2025V}.

In this paper, we will focus on the two dimensional Euclidean case where $VC(\mathcal{T}_{t}(\mathbbm{R}^2))=3$. One of the easy consequences of Theorem \ref{thm: VC dimension, main theorem} below is the fact that if $E \subseteq \mathbbm{R}^2$ satisfies $\overline{\delta}(E)>0$, then $VC(\mathcal{T}_{t}(E))=3$, for all sufficiently large $t$. Therefore, having  positive upper Banach density seems to be another suitable measure of largeness in the Euclidean setting, if one is interested in the aforementioned Ramsey theory-type questions. As with the generalization of Szek\'ely's problem, we are interested in generalizing this claim to two sets. Our results could very well be stated with circles, but it turns out that we can be even more general and can replace translates of the circle by translates of the curve $\Gamma$ which is essentially circle-like. {More precisely we let $\Gamma$ be a smooth, closed, and centrally symmetric planar curve with non-vanishing curvature. In addition, we require that $\Gamma$ is the boundary of a convex and compact region in the plane with a non-empty interior}. So, we study translates of $t\Gamma$ by vectors from a measurable set $B \subseteq \mathbbm{R}^2$, with the additional caveat that we do not consider the entirety of the curve $b +t\Gamma$, but only a portion of its arc, which lies inside a fixed measurable set $A$. Thus, for a fixed $t>0$, we consider the family
\[
\mathcal{T}_t(A,B):=\{(b+ t\Gamma) \cap A : b\in B\}.
\]
It is quite easy to see that our family can be empty for certain choices of $A$ and $B$, so we are yet again interested in finding a condition on $A$ and $B$ which will guarantee that for all sufficiently large $t>0$, not only are the families $\mathcal{T}_t(A,B)$ non-trivial, but they have the maximal possible $VC$ dimension, which as another consequence of Theorem \ref{thm: VC dimension, main theorem} below turns out to be $3$. We now present a quantity $\overline{\delta}_{VC}(A,B)$, positivity of which is sufficient if one wishes to obtain the desired Ramsey theory-type result. 
We start by introducing some helpful notation. For measurable sets $A,B \subseteq \mathbbm{R}^2$ and for $x,v_1,v_2,v_3,s_1,s_2,s_3 \in \mathbbm{R}^2$, we define $\mathcal{F}_{A,B}(x;v_1,v_2,v_3;s_1,s_2,s_3)$ as the product
\begin{align*}
&\mathbbm{1}_{B}(x)
\mathbbm{1}_{A}(x+v_1)\mathbbm{1}_{A}(x+v_2)\mathbbm{1}_{A}(x+v_3)\\
    &\mathbbm{1}_{B}(x+v_1+v_2)\mathbbm{1}_{B}(x+v_1+v_3)\mathbbm{1}_{B}(x+v_2+v_3)\\
    &\mathbbm{1}_{B}(x+v_1+s_1)\mathbbm{1}_{B}(x+v_2+s_2)\mathbbm{1}_{B}(x+v_3+s_3).
\end{align*}
This product corresponds to a particular placement of vertices of a flexible configuration from Theorem \ref{thm: VC dimension, configuration variant qualitative} below between sets $A$ and $B$; we advise the reader to look at Figure \ref{fig:Config}.
\begin{figure}
    \centering   \includegraphics[width=0.85\linewidth]{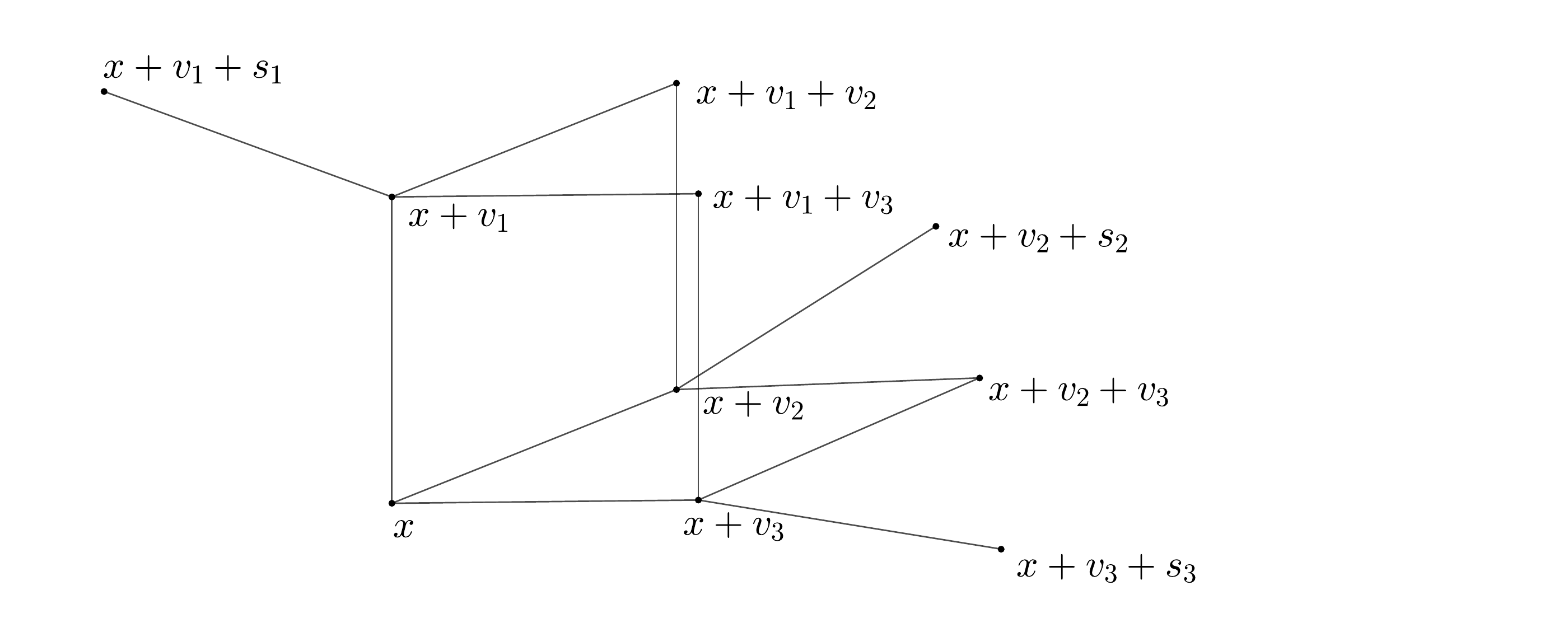}
    \caption{Flexible configuration from Theorem \ref{thm: VC dimension, configuration variant qualitative}}
    \label{fig:Config}
\end{figure}
\begin{defn}\label{def: density nearness condition}
For two measurable sets $A,B \subseteq \mathbbm{R}^2$, we define
\begin{align*}
        \overline{\delta}_{VC}(A,B) := \sup_{M \geq1}\!\limsup_{R \to +\infty} \sup_{z \in \mathbbm{R}^2} \inf_{M \leq r \leq 2^{-6}R}
    \frac{1}{(2r)^{12}R^2}\!\int_{z+[0,R]^2}\!\int_{([-r,r]^2)^6}
    &\!\mathcal{F}_{A,B}(x;v_1,v_2,v_3;s_1,s_2,s_3) \\
&\,\textup{d}v_1\,\textup{d}v_2\,\textup{d}v_3\,\textup{d}s_1\,\textup{d}s_2\,\textup{d}s_3\,\textup{d}x.
\end{align*}
\end{defn}
 Another consequence of Proposition \ref{prop: New condition comparability} below tells us that for a measurable set $E\subseteq \mathbbm{R}^2$, we have the chain of equivalences 
\[
\overline{\delta}(E)>0 \Longleftrightarrow \overline{\delta}_{VC}(E,E)>0\Longleftrightarrow \overline{\delta}_{VC}(\mathbbm{R}^2,E)>0\Longleftrightarrow \overline{\delta}_{VC}(E,\mathbbm{R}^2)>0,
\]
which tells us that the qualitative assumption of quantity from Definition \ref{def: density nearness condition} being positive for $A=B=E$ is equivalent to the condition of $E$ being a set of positive upper Banach density, so our new quantity is in a certain sense also a generalization of upper Banach density. Our second main {result} is
\begin{theorem}\label{thm: VC dimension, main theorem}
 Let $A,B \subseteq \mathbbm{R}^2$ be measurable and such that $\overline{\delta}_{VC}(A,B)>0$. { Let $\Gamma$ be a planar curve of non-vanishing curvature, which is closed, smooth, centrally symmetric. Assume also that $\Gamma$ is the boundary of some convex region in the plane, with a non-empty interior.} There exists a critical scale $t_0>0$ such that, for all $t \geq t_0$, it follows that $\mathcal{T}_t(A,B)$ has the Vapnik--Chervonenkis dimension equal to $3$.
\end{theorem}
 With the help of elementary results from differential geometry, which we briefly recall in Section \ref{sec: Geometry of general curves in the plane}, we show that Theorem \ref{thm: VC dimension, main theorem} reduces to the following theorem about flexible configurations in $\mathbbm{R}^2$. This theorem is analogous to the distance graph result from \cite{KP2023} and which would, by minor modifications, already imply a version of Theorem \ref{thm: VC dimension, main theorem} in case $A=B$ and $\Gamma$ being the unit circle.
\begin{theorem}[VC dimension, configuration variant]\label{thm: VC dimension, configuration variant qualitative}
  Let $A,B \subseteq \mathbbm{R}^2$ be measurable sets which satisfy $\overline{\delta}_{VC}(A,B)>0$. Let $\Gamma$ be a planar curve { of non-vanishing curvature}, which is closed, smooth, centrally symmetric and which is the boundary of some convex region in the plane, with a non-empty interior. There exists a critical scale $t_0 > 0$ such that, for all $t \geq t_0$, there exists a point $x \in B$ and vectors $v_1,v_2,v_3,s_1,s_2,s_3 \in t\Gamma$, such that
  \begin{equation}\label{eq: graph points which are translations}
    x,x+v_1+v_2,x+v_1+v_3,x+v_2+v_3,x+v_1+s_1,x+v_2+s_2,x+v_3+s_3 \in B
  \end{equation}
  and
  \begin{equation}\label{eq: graph points which shatter}
   x+v_1,x+v_2,x+v_3 \in A
 \end{equation}
and that all $10$ points listed above are mutually distinct.
\end{theorem}
To prove Theorem \ref{thm: VC dimension, configuration variant qualitative}, we follow the approach which has its roots in the work of Bourgain from \cite{B86:roth} and which was further developed by Cook, Magyar and Pramanik in \cite{CMP15:roth}. Namely we shall introduce a certain integral counting form, which we strive to bound from below by approximating it with its suitably smoothed out versions. This approach was further polished and applied in a series of papers by Durcik, Kova\v{c} and Rimani\'c, such as \cite{DKR18}, \cite{DK21}, \cite{DK22}, \cite{K20:anisotrop} and in a series of papers by Lyall and Magyar such as \cite{LM16:prod}, \cite{LM19:hypergraphs}, \cite{LM20}.
We emphasize that the key part of the aforementioned approximation is a bound for a particular singular Brascamp--Lieb form. The study of such forms has seen significant development in the last two decades, starting from the pioneering work of Kova\v{c} in \cite{Kov12} and ranging through further improvements by Durcik, Kova\v{c}, Thiele and Slav\'{\i}kov\'{a} in a series of papers such as \cite{Dur15, Dur17, DT20,DST22}, where \cite{DST22} is currently the most general known result. In principle, the estimate we shall require follows from Lemma~2 from \cite{DT20} by considering certain skew projections, which would again lead to a somewhat elaborate argument. Therefore, we prefer to give a short self-contained proof, while fully acknowledging that the approach is by no means new or original.
\subsection{Organization of the paper}
We briefly comment on how the rest of the paper is organized. Section \ref{sec: Preliminaries} focuses on preliminary results and is subdivided into Subsection \ref{sec: Gaussians and Fourier analysis}, in which we recall basic facts about the Fourier transform and Gaussian function, and Subsection \ref{sec: Geometry of general curves in the plane}, in which we recall basic facts from differential geometry of smooth curves in the plane. In the latter subsection we also prove Lemma \ref{lem: topological translations along curves} which, as mentioned, will be the key link in showing that Theorem \ref{thm: VC dimension, configuration variant qualitative} implies Theorem \ref{thm: VC dimension, main theorem}. Section \ref{sec: New definition is a generalization} is devoted to the condition which is a simultaneous generalization of Definitions \ref{def: Joint density for 2 sets} and \ref{def: density nearness condition}. It allows us to prove Proposition \ref{prop: New condition comparability}, which shows that our new conditions indeed generalize the condition of a set having positive upper Banach density. Section \ref{sec: Szekely and motivation} is devoted to the proof of Theorem \ref{thm: two set Szekely qualitative}. Section \ref{sec: VC dimensions proofs} is concerned with the proof of Theorem \ref{thm: VC dimension, main theorem}.
\section{Preliminaries and notation}\label{sec: Preliminaries}
We shall write $A \lesssim_PB$ if the inequality $A \leq C_PB$ holds for some unimportant finite constant $C_P>0$, which is allowed to depend on a set of parameters $P$. Similarly, we write $A \gtrsim_P B$ if $B \lesssim_P A$ holds and we also write $A\sim_P B$ if both $A \lesssim_PB$ and $A \gtrsim_PB$ hold.
We shall write $\mathbbm{1}_A$ for the indicator function and $\mathcal{P}(A)$ for the power set of a set $A$ and write $\textup{Int}(A)$ for the interior of $A$.
Since both problems we study concern the Euclidean plane $\mathbbm{R}^2$, unless specified otherwise, the reader can assume that all our functions have $\mathbbm{R}^2$ as their domain. For $\eta \in \{0,1\}^n$ and $v \in (\mathbbm{R}^2)^n$ we often write
\[
\eta \cdot v := \eta_1v_1+\eta_2v_2+\dots+\eta_nv_n,
\]
where, even if not explicitly, we assume that all tuples, either of numbers or of vectors, have coordinates denoted by the same letter but with indices, such as for instance $\eta=(\eta_1,\eta_2,\dots,\eta_n)$ and $v=(v_1,v_2,\dots,v_n)$. We also use the same symbol to denote both the above product and the standard Euclidean inner product on $\mathbbm{R}^2$ and we trust that it will be clear from context which meaning we assign to it. We shall make use of the following ``freeze functions''. For positive integers $k \leq n$ and $i_{1},i_{2},\dots,i_{k} \in \left\{ 1,2,\dots,n \right\}$ mutually distinct and for $a_{1},a_{2},\dots,a_{k} \in \left\{ 0,1 \right\}$, we consider the function $C_{(i_{1},a_{1}),(i_{2},a_{2}),\dots,(i_{k},a_{k})}:\left\{ 0,1 \right\}^{n}\to \left\{ 0,1 \right\}^n$ defined in such a way that if
\[
C_{(i_{1},a_{1}),(i_{2},a_{2}),\dots,(i_{k},a_{k})}(\eta_{1},\eta_{2},\dots,\eta_{n})=(s_{1},s_{2},\dots,s_{n}),
\]
then for $j \in \left\{ 1,2,\dots,n \right\}$, it follows that
\[
s_{j} = \begin{cases}
a_{l}, \qquad & j=i_{l} \textup{ for some } l\in \left\{ 1,2,\dots,k \right\}, \\
\eta_{j}, \qquad &\textup{otherwise}.
\end{cases}
\]
{Thus, for instance, in the case $n=4$, we have that
\[
C_{(1,0),(3,1)}(\eta_1,\eta_2,\eta_3,\eta_4)=(0,\eta_2,1,\eta_4),
\]
for any tuple $(\eta_1,\eta_2,\eta_3,\eta_4) \in \{0,1\}^4$.
}

By $\textup{co}(v_1,v_2,\dots,v_n)$ we denote the convex hull of points $v_1,v_2,\dots,v_n \in \mathbbm{R}^2$.
\subsection{Gaussian functions and Fourier analysis}\label{sec: Gaussians and Fourier analysis}
We use the following normalization of the Fourier transform for an integrable function $f$:
\[
\widehat{f}(\xi):=\int_{\mathbbm{R}^2}f(x)e^{-2 \pi i x \cdot \xi} \,\textup{d}x.
\]
It is a well-known fact that under this normalization, if we assume that $f$ belongs to the Schwartz class $\mathcal{S}(\mathbbm{R}^2)$, then the Fourier reconstruction formula holds pointwise, that is to say
\[
f(x) = \int_{\mathbbm{R}^2} \widehat{f}(\xi)e^{2 \pi i x\cdot \xi} \,\textup{d}\xi,
\]
for all $x \in \mathbbm{R}^2$. Similarly, for a Borel measure $\mu$, we use the following normalization of the Fourier transform
\[
\widehat{\mu}(\xi):=\int_{\mathbbm{R}^2}e^{-2 \pi i x \cdot \xi} \,\textup{d}\mu(x).
\]For $\lambda>0$, an integrable function $f$ and a Borel measure $\mu$, we denote their respective dilations by
\[
f_{\lambda}(x):=\lambda^{-2}f(\lambda^{-1}x), \quad\mu_{\lambda}(E):=\mu(\lambda^{-1}E)
\]
and we define their convolution as
\[
(\mu \ast f)(x):=\int_{\mathbbm{R}^2} f(x-y) \,\textup{d}\mu(y).
\]
By $\mathbbm{g}$ we denote the standard Gaussian on $\mathbbm{R}^2$ and by $\mathbbm{k}$ we denote its Laplacian, while for $i=1,2$, by $\mathbbm{h}^{(i)}$ we denote its $i$-th partial derivative. That is to say
\[
\mathbbm{g}(x):=e^{-2\pi \lvert x \rvert^2}, \quad \mathbbm{h}^{(i)}(x):=\partial_i\mathbbm{g}(x),\quad\mathbbm{k}(x):=\Delta\mathbbm{g}(x).
\]
The key properties of the Gaussian functions used in the proofs below are the following Gaussian convolution identities, which appear in \citep[Section~5]{DK22}. Let $\alpha,\beta >0$, then
\begin{align} 
\mathbbm{g}_{\alpha} \ast \mathbbm{g}_{\beta} &= \mathbbm{g}_{\sqrt{ \alpha^{2}+ \beta^2 }}\label{eq: Convolution of two Gaussians},\\ 
\sum_{l=1}^2 \mathbbm{h}^{(l)}_{\alpha} \ast \mathbbm{h}^{(l)}_{\beta} &=\frac{\alpha \beta}{\alpha^2 + \beta ^2} \mathbbm{k}_{\sqrt{\alpha^2 + \beta^2}} , \label{eq: Convolution of partial derivatives of Gaussians}  \\ \mathbbm{k}_{\alpha} \ast \mathbbm{g}_{\beta} &=\frac{\alpha^2}{\alpha^2 + \beta^2} \mathbbm{k}_{\sqrt{\alpha^2 + \beta^2}}.\label{eq: Convolution of a gaussian Laplacian and a gaussian}
\end{align}
The identities follow by passing to the Fourier side and performing a simple calculation. We also make use of the following inequalities, which can be found in \citep[Section~5]{DK22}, as well. { For $t,s,r,\lambda>0$ which satisfy  $t\lambda \sim r \sim s$ and $t\lambda \geq r$, $1 \geq t\geq \varepsilon$, for some $\varepsilon>0$}, it follows that
\begin{align}
(\mu_{t}\ast \mathbbm{g}_{t\lambda})(x) &\lesssim \varepsilon^{-3}\int_{1}^{\infty} \mathbbm{g}_{s\gamma}(x)\, \frac{\textup{d}\gamma}{\gamma^2},\label{eq: Domination of Gaussian with t lambda} \\
(\mu_{t}\ast \mathbbm{g}_{r})(x) &\lesssim \varepsilon^{-3}\int_{1}^{\infty} \mathbbm{g}_{s\gamma}(x) \,\frac{\textup{d}\gamma}{\gamma^2}\label{eq: Domination of Gaussian with r},
\end{align}
where $\mu$ denotes a Borel probability measure. By direct calculation, we check that the reparametrized heat equation is satisfied by the Gaussian functions, i.e., it holds that
\begin{equation}\label{eq:heateq}
\frac{\partial}{\partial t} \mathbbm{g}_t(x) = \frac{1}{2\pi t} \mathbbm{k}_t(x).
\end{equation}
The following elementary lemma will play a crucial role in the proof of Theorem \ref{thm: two set Szekely qualitative}. The proof follows by combining an elementary dyadic decomposition of the positive real numbers and the trivial estimate for tails of Schwartz functions.
\begin{lem}[Essential disjointness of Gaussian functions]\label{lem: Smooth disjointness estimate}
Let $J \in \mathbbm{N}$ and let {$\zeta_1, \zeta_2, \dots, \zeta_J \in \mathbbm{R}^2$} be such that for all $j \in \{1,2,\dots,J-1\}$
\[
{2 \lvert \zeta_j \rvert \leq \lvert \zeta_{j+1} \rvert}.
\]
Then, the following bound, which we stress is uniform in the number of variables, holds
\[
\sum_{j=1}^J \lvert \widehat{\mathbbm{k}}({\zeta_j})\rvert \lesssim 1.
\]
\end{lem}
\begin{lem}[Fourier dimension estimate]\label{lem: Fourier dimension estimate}
Let $\mu$ be a probability measure supported on a closed, smooth curve of non-vanishing curvature, which is the boundary of some compact and convex set in the plane. Then, the Fourier transform of $\mu$ satisfies
\begin{itemize}
    \item [(a)] $\sup_{\xi \in \mathbbm{R}^2} \lvert \xi \rvert^{\frac{1}{2}} \lvert \widehat{\mu}(\xi) \rvert< +\infty.$
    \item [(b)] $\big \vert  \widehat{\mu}(\lambda \xi) \widehat{\mathbbm{k}}(t \lambda \xi)\big \vert \lesssim t^{\frac{1}{2}}$, for all $\xi \in \mathbbm{R}^2$, and all $t, \lambda>0$.
\end{itemize}   
\end{lem}
 The proof of (a) is quite standard and can be found as Theorem 1 in Section 8.3 in \cite{St93:book}, while (b) follows quite easily from (a); the interested reader may consult \cite{K20:anisotrop} or \cite{KP2023} for details. One measure that satisfies the conditions of the above lemma, and which we make frequent use of, is the uniform probability measure supported on the unit circle centered at the origin. We denote this measure by $\sigma$.

Lastly, we recall some basic tools from additive combinatorics. Let $f\colon\mathbbm{R}^2\to[0,1]$ be a compactly supported measurable function. For $n \in \mathbbm{N}$, we denote
\[ \mathcal{F}_n(x;v_1,\ldots,v_n) := \prod_{\eta\in\{0,1\}^n} f(x + \eta_1v_1 + \eta_2v_2+\dots+\eta_nv_n), \]
for $x,v_1,\ldots,v_n\in\mathbbm{R}^2$. These products are the basic building blocks of the so-called \emph{Gowers uniformity norms} $\Vert \cdot\Vert_{\textup{U}^n}$, which are defined as
\begin{equation}\label{eq:Gowersdef}
\Vert f\Vert_{\textup{U}^n(\mathbbm{R}^2)}:=\Big(\int_{(\mathbbm{R}^2)^{n+1}} \mathcal{F}_n(x;v_1,\ldots,v_n) \,\textup{d}x \,\textup{d}v_1 \cdots \textup{d}v_n\Big)^{\frac{1}{2^n}} .
\end{equation}
Uniformity norms are usually studied in the context of locally-compact Abelian groups and in this general setting, one can prove the so called Gowers--Cauchy--Schwarz inequality; the interested reader might look at \cite{Gow01,HK05} or \cite{ET12}. For our purposes, the following special case, the proof of which follows from inductively applying the Cauchy--Schwarz inequality in $\mathbbm{R}^2$, will suffice:
\begin{equation}\label{eq: GCS}
 \bigg\vert \int_{(\mathbbm{R}^2)^{n+1}} \prod_{\eta\in\{0,1\}^n} f_{\eta}(x + \eta \cdot v) \,\textup{d}x \,\textup{d}v_1 \cdots \textup{d}v_n \bigg\vert 
\leq \prod_{\eta\in\{0,1\}^n} \Vert f_{\eta}\Vert_{\textup{U}^n(\mathbbm{R}^2)},
\end{equation}
where $(f_{\eta})_{\eta \in \{0,1\}^n}$ are bounded, real-valued and compactly supported measurable functions.
By testing the above inequality with suitably chosen indicator functions of cubes (see \cite{KP2023} for more elaboration), we can conclude that
\begin{equation}\label{eq:GCScorr}
\Vert f\Vert_{\textup{U}^n(\mathbbm{R}^2)} \gtrsim_{n} \vert Q\vert ^{-1+(n+1)/2^n} \int_Q f(x)\,\textup{d}x,
\end{equation}
where $f$ is a positive function, supported on the cube $Q$.
\subsection{Geometry of curves in the plane}\label{sec: Geometry of general curves in the plane}
Throughout this section, let $\Gamma$ be a closed, smooth, centrally symmetric curve with non-vanishing curvature, which is the boundary of some compact and convex planar set with non-empty interior. Let $\mu$ be the {normalized (uniform)} arclength measure supported on $\Gamma$.
The following proposition is geometrically quite self-evident, however a rigorous proof requires some theory. The standard reference is a book by do Carmo \cite{doCarmo2016}, and the interested reader might consult Sections 1--7; chiefly Exercise 9\@ and 5--7; chiefly Proposition 1.
\begin{prop}\label{prop: Supporting line}
   { Let $\Gamma$ be a closed, smooth planar curve { of non-vanishing curvature}, which is the boundary of some compact set $C$ with a non-empty interior. Then, $C$ is convex if and only if, the so-called, ``supporting line'' property holds. Namely, for every point $a \in \Gamma$, there exists a unique vector $n_a\in\mathbbm{R}^2$ such that for all $b \in \Gamma \setminus \{a\}$, it follows that
    \[
    n_a \cdot (b-a)< 0.
    \]
    Geometrically, the supporting line property hold if the entirety of the curve $\Gamma$, except for point $a$, is entirely contained in one of the two open half-planes determined by the line $\ell_a:=a+\{n_a\}^\perp$.}
\end{prop}
\begin{lem}\label{lem: Curves and lines intersect at 2 points}
Every line $l$ in $\mathbbm{R}^2$ intersects $\Gamma$ in at most two points.
\end{lem}
\begin{proof}
    Assume that $l$ intersects $\Gamma$ in three mutually distinct points $x,y$ and $z$. Without loss of generality, we can assume that there exists some $t\in(0,1)$ such that $z=tx+(1-t)y.$ Since $z$ belongs to $\Gamma$, Proposition \ref{prop: Supporting line} implies that there exists some $n \in \mathbbm{R}^2$ such that for all $b\in \Gamma \setminus \left\{ z \right\}$, it follows that $n \cdot (b-z) <0$. In particular, this leads to a contradiction:
    \[0=n \cdot((tx+(1-t)y) - z)=t(n \cdot (x-z))+(1-t)(n \cdot(y-z))<0.\qedhere\]
\end{proof}
\begin{lem}\label{lem: Critical convexity}
Let $a_{1},a_{2},a_{3},b_{1},b_{2},b_{3} \in \mathbbm{R}^2$ be mutually distinct points which satisfy
\[
b_{1}-a_{1}=b_{2}-a_{2}=b_{3}-a_{3}.
\]
Then one of these points lies in the convex hull spanned by some three of the remaining points.
\end{lem}
\begin{proof}
By making an appropriate change of coordinates, the claim reduces to checking three distinct cases which correspond to different arrangements of the given points and which are illustrated in Figure \ref{fig:conv5}.
\begin{figure}
\includegraphics[width=0.9\linewidth]{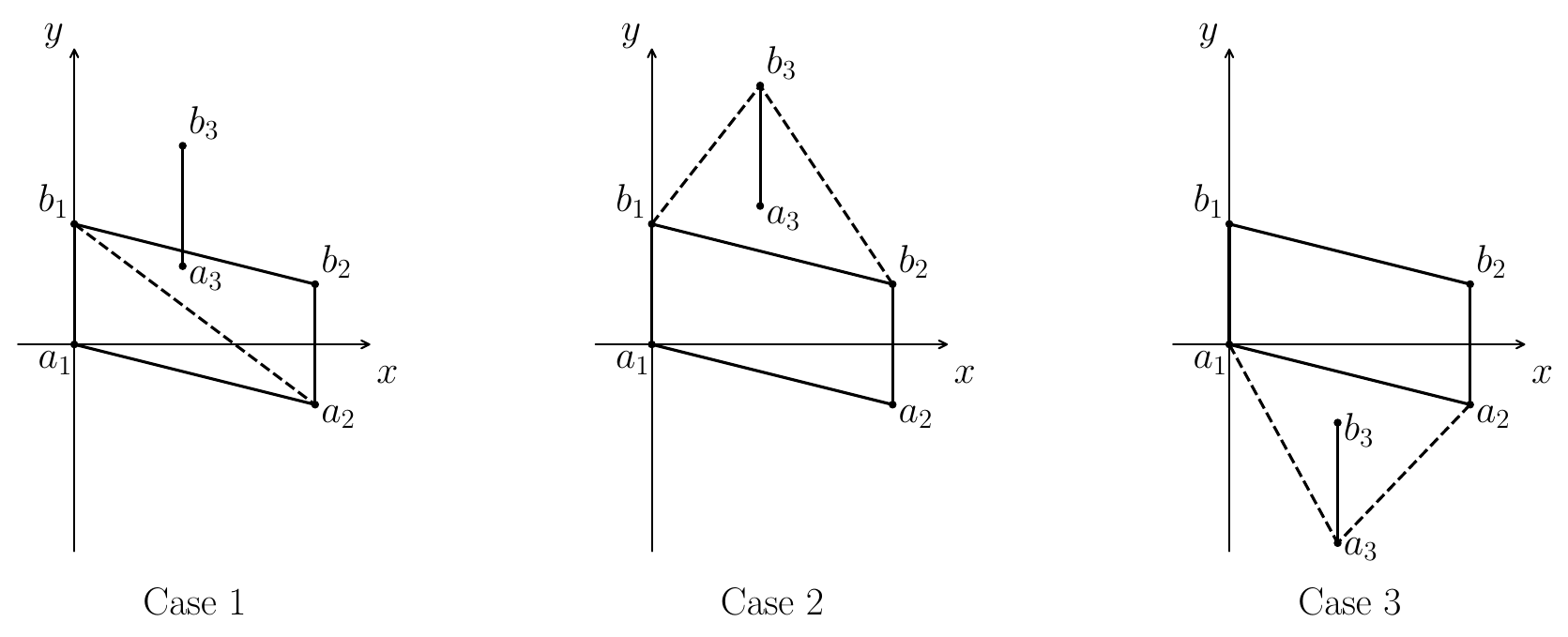}
\centering
\caption{Sketch of proof of Lemma \ref{lem: Critical convexity}}
\label{fig:conv5}
\end{figure}   
\end{proof}

As mentioned in the introduction, the following lemma is crucial for the reduction of Theorem \ref{thm: VC dimension, main theorem} to Theorem \ref{thm: VC dimension, configuration variant qualitative}.
\begin{lem}\label{lem: topological translations along curves}
Let $\Gamma \subseteq \mathbbm{R}^2$ be as above.
\begin{itemize}
    \item[(a)] For any mutually distinct $v_1,v_2 \in \mathbbm{R}^2$, the intersection $(v_1+\Gamma) \cap (v_2 +\Gamma)$ consists of at most two points.
    \item[(b)] For two mutually distinct points $x,y \in \mathbbm{R}^2$, there exist at most two translates of the curve $\Gamma$ which contain both $x$ and $y$.
\end{itemize}
\end{lem}
\begin{proof}[Proof of (a)]
 Assume that there exist at least three distinct intersections $x,y,z \in (v_{1} + \Gamma) \cap (v_{2} + \Gamma)$. Then, we can write
\begin{equation}\label{eq: SixPointsOnACurve}
x=v_{1}+\gamma_{1,x}=v_{2} + \gamma_{2,x},  \qquad
y=v_{1} +\gamma_{1,y}=v_{2} + \gamma_{2,y}, \qquad
z=v_{1} +\gamma_{1,z}=v_{2} +\gamma_{2,z},
\end{equation}
for some $\gamma_{1,x},\gamma_{1,y},\gamma_{1,z},\gamma_{2,x},\gamma_{2,y},\gamma_{2,z} \in \Gamma$. We claim that at least two of these points are equal. Assume on the contrary that all $6$ points are mutually distinct. By Lemma \ref{lem: Critical convexity}, it follows that there exist $a,b,c,d \in \left\{ \gamma_{1,x},\gamma_{1,y},\gamma_{1,z},\gamma_{2,x},\gamma_{2,y},\gamma_{2,z} \right\}$ all mutually distinct and such that $a\in\textup{co}(b,c,d)$.

Firstly, we notice that it necessarily follows that $a \in \textup{Int}(\textup{co}(b,c,d))$, because if $a$ was an element from the boundary of the convex hull, this would imply the existence of the line which intersects $\Gamma$ in three points, which {cannot} exist because of Lemma \ref{lem: Curves and lines intersect at 2 points}. For instance if $a \in [b,d]$, then it lies on the line spanned by $b$ and $d$ and we obtain our contradiction.

Let $\ell_a$ be the supporting line to $\Gamma$ at $a$, which we know exists because of Proposition \ref{prop: Supporting line}, and let $\mathcal{H}$ be an open half-plane which contains $\Gamma\setminus\{a\}$. However, since $\mathcal{H}$ is a convex set, it follows that
\[
a \in \textup{Int}(\textup{co}(b,c,d))\subseteq \textup{co}(b,c,d)\subseteq \mathcal{H},
\]
which is in contradiction with Proposition \ref{prop: Supporting line}.

Therefore, we know that the set $\left\{ \gamma_{1,x},\gamma_{1,y},\gamma_{1,z},\gamma_{2,x},\gamma_{2,y},\gamma_{2,z} \right\}$ consists of at most $5$ points. Thus, two of the elements listed above actually equal each other. We consider several possibilities how this might occur and show that they all lead to a contradiction. If for $a \in \left\{ x,y,z \right\}$ it follows that $\gamma_{1,a}=\gamma_{2,a}$, then by \eqref{eq: SixPointsOnACurve}, we conclude that $v_1=v_2$,
which contradicts the assumption that $v_{1}$ and $v_{2}$ are mutually distinct. 
If for mutually distinct $a,b \in \left\{ x,y,z \right\}$ it follows that $\gamma_{1,a}=\gamma_{2,b}$, again by \eqref{eq: SixPointsOnACurve}, we would get that
\[
\gamma_{1,a}=\frac{1}{2}(\gamma_{1,b} +\gamma_{2,a}),
\]
which would imply that the line determined by $\gamma_{1,b}$ and $\gamma_{2,a}$ also contains $\gamma_{1,a}$, which contradicts Lemma \ref{lem: Curves and lines intersect at 2 points}, since all three points $\gamma_{1,a},\gamma_{1,b}$ and $\gamma_{2,a}$ lie on the curve $\Gamma$.
Finally, if for mutually distinct $a,b \in \left\{ x,y,z \right\}$ it either follows that$\gamma_{1,a}=\gamma_{1,b}$ or $\gamma_{2,a}=\gamma_{2,b}$, then, once more, \eqref{eq: SixPointsOnACurve} implies that either $\gamma_{2,a}=\gamma_{2,b}$ or $\gamma_{1,a}=\gamma_{1,b}$ and hence in both cases $a=b$, contradicting our assumption. Since we have exhausted all the possible cases of equality, this concludes the proof.
\end{proof}
\begin{proof}[Proof of (b)]
Assume the contrary, that is, that there exist mutually distinct vectors $v_{1},v_{2},v_{3} \in \mathbbm{R}^2$ such that
        \[
        x,y \in v_{1} + \Gamma, v_{2} + \Gamma, v_{3}+\Gamma.
        \]
        By central symmetry of $\Gamma$, we conclude that
        \[
        v_{1},v_{2},v_{3} \in (x+\Gamma) \cap (y+\Gamma).
        \]
        However, this is a contradiction with (a), because we have assumed that $v_{1},v_{2},v_{3}$ are mutually distinct.
\end{proof}
\section{Common generalization of $\overline{\delta}(A,B)$ and $\overline{\delta}_{VC}(A,B)$ }\label{sec: New definition is a generalization}
In this section we explore the common thread between Definitions \ref{def: Joint density for 2 sets} and \ref{def: density nearness condition}. In fact, we will study a slightly more general quantity, which we now define.
\begin{defn}[Joint upper Banach density]\label{def: Joint Upper Banach Density}
For $n \in \mathbb{N}$ and $2^n$ locally integrable functions $(f_{\eta})_{\eta \in \left\{ 0,1 \right\}^n}$, we consider the quantity
\begin{align}\label{eq: New condition}
\overline{\delta}((f_{\eta})_{\eta \in \left\{ 0,1 \right\}^n}):=\sup_{M \geq1} \limsup_{R \to +\infty} \sup_{z \in \mathbbm{R}^2} \inf_{M \leq r \leq 2^{-n}R}
    \frac{1}{(2r)^{2n}R^{2}}\int_{z+[0,R]^2}\int_{([-r,r]^2)^n}&\mathcal{F}_{n}(x;v_{1},\dots,v_{n})\,\notag\\
    &\textup{d}v_{1}\dots\,\textup{d}v_{n}\,\textup{d}x,
\end{align}
where for $x,v_{1},\dots,v_{n} \in \mathbbm{R}^2$, we have slightly abused notation introduced in Section \ref{sec: Gaussians and Fourier analysis}, by defining
\[
\mathcal{F}_{n}(x;v_{1},\dots,v_{n}):=\prod_{\eta\in \left\{ 0,1 \right\}^n}f_{\eta}(x+\eta_{1}v_{1}+\dots+\eta_{n}v_{n}).
\]    
\end{defn}
We note that in the definition of upper Banach density we are essentially calculating limits of averages of an indicator function over a given cube. The above quantity 
generalizes this idea by looking at averages of multiple functions, whose variables are tied up in a hypercubic structure. It is easy to see that by picking various different indicator functions, we are essentially considering versions of upper Banach density for multiple sets. Indeed, for measurable sets $A,B \subseteq \mathbbm{R}^2$ and for $n=1$, we recover Definition \ref{def: Joint density for 2 sets}, i.e.,
\[
\overline{\delta}(A,B)=\overline{\delta}(\mathbbm{1}_A,\mathbbm{1}_B).
\]
We also briefly comment that Definition \ref{def: Joint Upper Banach Density} generalizes Definition \ref{def: density nearness condition}
in the case when $n=6$. More precisely, if in Definition \ref{def: density nearness condition} we respectively write
$v_4,v_5,v_6$ instead of $s_1,s_2,s_3$, and if we set
\begin{equation}\label{eq: VCChoiceOfVertices}
\begin{aligned}
f_{(1,0,0,0,0,0)}:=&\,
f_{(0,1,0,0,0,0)}=
f_{(0,0,1,0,0,0)}=\mathbbm{1}_{A},\\
f_{(0,0,0,0,0,0)}:=&\,f_{({1},0,0,1,0,0)}=f_{(0,{1},0,0,1,0)}=f_{(0,0,{1},0,0,1)}\\
=&\,f_{(1,1,0,0,0,0)}=f_{(1,0,1,0,0,0)}=f_{(0,1,1,0,0,0)}=\mathbbm{1}_{B}
\end{aligned}
\end{equation}
and for all other tuples $(\eta_{1},\eta_{2},\eta_{3},\eta_{4},\eta_{5},\eta_{6})$ not listed here, we set $f_{(\eta_1,\eta_2,\dots,\eta_6)}=\mathbbm{1}_{\mathbbm{R}^2}$, then we can notice that
\[
\overline{\delta}((f_\eta)_{\eta \in \{0,1\}^6})=\overline{\delta}_{VC}(A,B).
\]
We prefer to use both the full hypercube notation and the notation which highlights the vertices of interest to us, depending on the context. We turn our attention to showing the proposition which helped us conclude that the qualitative conditions of our new quantities being positive truly generalizes the qualitative condition of a set having  positive upper Banach density.
\begin{prop}\label{prop: New condition comparability}
Consider a measurable set $A \subseteq \mathbbm{R}^2$. For $n \in \mathbb{N}$, also consider $2^n$ locally integrable functions $(f_{\eta})_{{\eta} \in \left\{ 0,1 \right\}^n}$, where for each ${\eta} \in \{0,1\}^n$, $f_{\eta}\in \{\mathbbm{1}_A,\mathbbm{1}_{\mathbbm{R}^2}\}$. Assume that at least one $f_{\eta}$ equals $\mathbbm{1}_{A}$.
We claim that
\[
\overline{\delta}(A)^{\frac{n+1}{2^n}}\gtrsim_{n}\overline{\delta}((f_{\eta})_{{\eta} \in \left\{ 0,1 \right\}^n})\gtrsim_{n}\overline{\delta}(A)^{2^n}.
\]
\end{prop}
\begin{proof}
 From the definition of the upper Banach density, we conclude that there exists a sufficiently large $R_{0}>0$ such that for all $R \geq R_{0}$ and all $z \in \mathbbm{R}^2$, we have that 
\begin{equation}\label{eq: ComparingWithBanachUpperBound}
\frac{\lvert A \cap (z+[0,R]^2) \rvert }{R^2}<2\overline{\delta}(A),    
\end{equation}
Fix $M \geq 1$ and $z \in \mathbbm{R}^2$ and denote
\[
I(R,z):=\inf_{M\leq r\leq 2^{-n}R}\frac{1}{(2r)^{2n}R^{2}}\int_{z+[0,R]^2}\int_{([-r,r]^2)^n}\mathcal{F}_{n}(x;v_{1},v_{2},\dots,v_{n})\,\textup{d}v_{1}\dots \textup{d}v_{n}\,\textup{d}x.
\]
By considering $r=2^{-n}R$, we easily conclude that
\[
I(R,z) \lesssim_{n} \frac{1}{(R^2)^{n+1}}\int_{(\mathbbm{R}^2)^{n+1}}\prod_{{\eta} \in \{0,1\}^n}(\mathbbm{1}_{(z-(R,R))+[0,{3}R]^2}\cdot f_{\eta})(x+{\eta}\cdot v)\,\textup{d}v_{1}\dots \textup{d}v_{n}\,\textup{d}x.
\]
Since at least one function in the above product is equal to $\mathbbm{1}_{A\cap((z-(R,R))+[0,{3}R]^2)}$, we can bound all the functions on the remaining vertices of the hypercube by $\mathbbm{1}_{(z-(R,R))+[0,{3}R]^2}$. After performing this bound and applying the Gowers--Cauchy--Schwarz inequality \eqref{eq: GCS}, followed by the assumption \eqref{eq: ComparingWithBanachUpperBound}, we conclude that
\begin{align*}
I(R,z) &\lesssim_{n}R^{-2n-2}\lVert \mathbbm{1}_{A \cap ((z-(R,R))+[0,{3}R]^2)} \rVert_{\textup{U}^{n}(\mathbbm{R}^2)}\lVert \mathbbm{1}_{(z-(R,R))+[0,{3}R]^2} \rVert_{\textup{U}^n(\mathbbm{R}^2)}^{2^{n}-1} \\
&\lesssim R^{-2n-2}R^{\frac{n+1}{2^{n-1}}}\!\Bigg(\frac{\lvert A\cap((z-(R,R))+[0,{3}R]^2) \rvert }{({3}R)^2}\Bigg)^{\frac{n+1}{2^n}}\!\!\!\!\!\lvert ((z-(R,R))+[0,{3}R]^2) \rvert^{\frac{n+1}{2^n}(2^n-1)} \\
&\lesssim\overline{\delta}(A)^{\frac{n+1}{2^n}},\\
\end{align*}
where we have used the fact that the uniformity norms of indicator functions of measurable sets can be bounded by appropriate powers of the measure of the set. Since this was true for all $z \in \mathbbm{R}^2$ and $M \geq 1$, by taking a supremum in $z$, followed by a limit superior in $R$ and finally a supremum in $M$, we conclude that
\[
\overline{\delta}((f_{\eta})_{{\eta} \in \left\{ 0,1 \right\}^n})\lesssim_{n}\overline{\delta}(A)^{\frac{n+1}{2^n}}.
\]
Next, we prove the lower bound. Unpacking the definition of the usual upper Banach density, we can find an increasing sequence of positive real numbers $(R_m)_{m \in \mathbbm{N}}$ converging to infinity and a sequence $(z_m)_{m \in \mathbb{N}}$ in $\mathbbm{R}^2$ such that for all $m \in \mathbb{N}$, we have
\begin{equation}\label{eq: ComparingWithBanachLowerBound}
\int_{z_m+[0,R_m]^2} \mathbbm{1}_A(x) \textup{d}x > \frac{\overline{\delta}(A)R_m ^2}{2}.    
\end{equation}
Now fix some $M \geq 1$ and consider $m \in \mathbb{N}$ large enough so that $M<2^{-n}R_m$.
For $M \leq r \leq 2^{-n}R_m$, we strive to bound
\begin{align*}
I(r,R_m,z_m):=\frac{1}{(2r)^{2n}R_m^2}\int_{z_m+[0,R_m]^2}\int_{([-r,r]^2)^n}
\mathcal{F}_{n}(x;v_{1},v_{2},\dots,v_{n})
\,\textup{d}v_{1}\dots \textup{d}v_{n}\,\textup{d}x
\end{align*}
from below by a positive constant multiple of $\overline{\delta}(A)^{2^n}$. To start with, we partition the cube $z_m+[0,R_m]^2=(z_1^m,z_2^m)+[0,R_m]^2$ into cubes of side-length $r'$, where
\[
r':=\frac{1}{2}\frac{R_m}{\left\lfloor \frac{R_m}{r} \right \rfloor}.
\]
From elementary properties of the floor function, it follows that $
\frac{r}{2}\leq {r'} \leq r.
$
Explicitly, we consider the family
\[
\mathcal{Q} :=\{[z_1^m+k{r'},z_1^m+(k+1){r'})\times [z_2^m+l{r'},z_2^m+(l+1){r'}) :{ k,l\in \{0,1\dots,2\left\lfloor\frac{R_m}{r}\right\rfloor-1 \}}\},
\]
noting that $\lvert\mathcal{Q}\rvert=4\left\lfloor\frac{R_m}{r}\right\rfloor^2$. 
We begin by estimating all the functions in the product $\mathcal{F}_{n}(x;v_{1}\dots,v_{n})$ from below by $\mathbbm{1}_A$. We also note that if we know that all $2^n$ points $x+ {\eta} \cdot v$ lie inside the same cube $Q$ of side-length ${r'}$, then the edges $v_{i}$ lie inside the cube {$[-r',r']^2$} and so we can actually estimate
\begin{align*}
I(r,R_m,z_{m})&\gtrsim_{n}\frac{1}{r^{2n}R_m^2 } \sum_{Q \in \mathcal{Q}}\lVert \mathbbm{1}_{A \cap Q} \rVert_{\textup{U}^n(\mathbbm{R}^2)}^{2^n}.
\end{align*}
Combining our inequality derived from Gowers--Cauchy--Schwarz \eqref{eq:GCScorr} with Jensen's inequality and then applying \ref{eq: ComparingWithBanachLowerBound}, we obtain
\begin{align*}
I(r,R_m,z_m)
&\gtrsim_{n}\frac{r^{-2^{n+1}+2}\lvert \mathcal{Q} \rvert }{R_{m}^2}\Bigg(\frac{1}{\lvert \mathcal{Q} \rvert }\sum_{Q \in \mathcal{Q}}\int_{Q} \mathbbm{1}_{A}(x)\,\textup{d}x\Bigg)^{2^n} \gtrsim_n\overline{\delta}(A)^{2^n}.
\end{align*}
Since this was true for all $M \geq 1$, $m \in \mathbb{N}$ large enough and $M \leq r \leq 2^{-n}R_m$, we can first take the infimum in $r$ and then take the limit superior in $R$ and finally a supremum in $M$ to conclude that
\[
\overline{\delta}((f_{\eta})_{{\eta} \in \left\{ 0,1 \right\}^n})\gtrsim_{n} \overline{\delta}(A)^{2^n}.
\]
\end{proof}
\section{Proof of Theorem \ref{thm: two set Szekely qualitative}}\label{sec: Szekely and motivation}
We now turn to proving Theorem \ref{thm: two set Szekely qualitative}. In the spirit of Bourgain, we reduce this qualitative theorem to the following more quantitative formulation.
\begin{theorem}[Two-set Szek\'ely, quantitative version]\label{thm: two set Szekely quantitative}
    {Let $\delta \in (0,1)$ be arbitrary. Then, there exists $J=J(\delta)$ such that for all $\lambda_1,\lambda_2,\dots,\lambda_J >1$, $R>0$ and $M\geq 1$ which satisfy 
    \begin{itemize}
        \item $2\lambda_k \leq \lambda_{k+1}$ for all $k=1,2,\dots,J-1$,
        \item $R \geq 2\lambda_J$,
        \item $\lambda_1,\lambda_2,\dots,\lambda_J \in [M,2^{-1}R]$,
    \end{itemize}
    the following holds:
    For an arbitrary $z\in \mathbbm{R}^2$ and for all $A,B \subseteq z+[0,R]^2$ which satisfy
    \begin{equation}\label{eq: SzekelyAssumption}
    \inf_{M \leq r \leq 2^{-1}R} \frac{1}{(2r)^2}\int_{z+[0,R]^2}\int_{\mathbbm{R}^2} \mathbbm{1}_A(x)\mathbbm{1}_B(y)\mathbbm{1}_{[-r,r]^2}(x-y) \,\textup{d}y\,\textup{d}x > \delta R^2,
    \end{equation}
    there exists $j \in \{1,2,\dots,J\}$ such that for $\lambda=\lambda_j$, it follows that
    \begin{equation}\label{eq: SzekelyConclusion}
    \int_{\mathbbm{R}^2}\int_{\mathbbm{R}^2}\mathbbm{1}_A(x)\mathbbm{1}_B(x+v)\, \textup{d}\sigma_{\lambda}(v)\,\textup{d}x \gtrsim_{\delta} R^2.
    \end{equation}}
\end{theorem}
\begin{proof}[Theorem \ref{thm: two set Szekely quantitative}\ implies Theorem \ref{thm: two set Szekely qualitative}]
    Assume that Theorem \ref{thm: two set Szekely quantitative} is true and that Theorem \ref{thm: two set Szekely qualitative} is not. Then, there exist measurable sets $A, B \subseteq \mathbbm{R}^2$ such that $\overline{\delta}(A,B)>0$ and there exists a sequence of scales $(\lambda_m)_{m \in \mathbbm{N}}$ such that, for all $x \in A$ and $y \in B$, $\lvert x- y \rvert \neq \lambda_m$ for all $m \in \mathbbm{N}$. By passing to a subsequence, we can assume that $\lambda_1 > 1$ and $2\lambda_m \leq \lambda_{m+1}$ for all $m \in \mathbbm{N}$. We now apply Theorem \ref{thm: two set Szekely quantitative} with $\delta=\frac{\overline{\delta}(A,B)}{4}$ to obtain $J\in \mathbbm{N}$, for which the conclusion of Theorem \ref{thm: two set Szekely quantitative} holds. 
    Unpacking the definition of $\overline{\delta}(A,B)$, we find $M \geq 1$, a {sequence} $(R_m)_{m \in \mathbbm{N}}$ of positive real numbers converging to infinity, and a sequence $(z_m)_{m \in \mathbbm{N}}$ in $\mathbbm{R}^2$. For all $m \in \mathbbm{N}$, it follows that
    \[
    \inf_{M \leq r \leq 2^{-1}R_m} \frac{1}{4r^2}\int_{z_m+[0,R_m]^2}\int_{\mathbbm{R}^2} \mathbbm{1}_{A}(x)\mathbbm{1}_{B}(y)\mathbbm{1}_{[-r,r]^2}(x-y)\,\textup{d}y\,\textup{d}x >\delta R_m^2.
    \]
    Since $(\lambda_m)_{m \in \mathbbm{N}}$ and $(R_m)_{m \in \mathbbm{N}}$ are sequences converging to infinity, there exist some $i,m \in \mathbbm{N}$ such that $\lambda_i>M$ and $2^{-1}R_m>2\lambda_{i+J-1}$. Clearly, all the hypotheses of Theorem \ref{thm: two set Szekely quantitative} are now satisfied and so, we can find some $j \in \{0,1,\dots,J-1\}$ such that \eqref{eq: SzekelyConclusion} holds for $\lambda=\lambda_{i+j}$.
    This however immediately implies the existence of some $x \in A$ and $y \in B$ such that
    \[
    \lvert x- y \rvert = \lambda_{i+j},
    \]
    contradicting the very definition of the scale $\lambda_{i+j}$.
\end{proof}

\begin{proof}[Proof of Theorem \ref{thm: two set Szekely quantitative}]
For logical rigor's sake, we immediately choose
\begin{equation}\label{eq: TwoSetSzekelyParameterChoices}
\varepsilon:=\Big(\frac{\delta {{C_{\textup{str}}}}}{3{C_{\textup{uni}}}}\Big)^2 \qquad \textup{and then} \qquad J:=\left\lceil\frac{3 {C_{\textup{err}}} \log(\varepsilon^{-1})}{{C_{\textup{str}}}{\delta}}\right\rceil, 
\end{equation}
where ${C_{\textup{str}}} := \inf_{z \in B(0,3)} \mathbbm{g}(z)$, ${C_{\textup{uni}}}$ is the implicit constant from Lemma \ref{lem: Fourier dimension estimate} (b) and ${C_{\textup{err}}}$ is the implicit constant from Lemma \ref{lem: Smooth disjointness estimate}. 
Let $\lambda_1, \lambda_2,\dots,\lambda_J >1$ be such that $2\lambda_k \leq \lambda_{k+1}$ for all $k = 1,2,\dots,J-1$ and also let $R,M > 1$ be such that $\lambda_1, \lambda_2,\dots,\lambda_J \in [M,2^{-1}R]$ and $R>2\lambda_{J}$. Consider $z \in \mathbbm{R}^2$ and $A,B \subseteq z+[0,R]^2$ which satisfy \eqref{eq: SzekelyAssumption}. For clarity's sake, we allow {ourselves} the flexibility to use the symbol $\lambda$ to stand for any one of the scales $\lambda_1, \lambda_2, \dots, \lambda_J$ when the estimate in question does not hinge on precisely which one of them we are working with.
To simplify notation, we consider the counting form
\[ \mathcal{N}^{0}_{\lambda} :=
\int_{\mathbbm{R}^{2}} \int_{\mathbbm{R}^2} \mathbbm{1}_A(x)\mathbbm{1}_B(x+v) \,\textup{d}\sigma_{\lambda}(v) \,\textup{d}x. \]
We briefly mention that Bourgain's original proof of Szek\'{e}ly's theorem in \cite{B86:roth}  proceeds by taking the Fourier transform and splitting the integral analogous to the one above into ``low'', ``high'' and ``mid'' frequency parts. The interested reader may also wish to consult \cite{Tao2021_BourgainToolkit}.
However, we rather choose to follow a variant of the scheme introduced by Cook, Magyar and Pramanik in \cite{CMP15:roth}, which introduces an additional parameter, which we call the ``smoothness scale''. More precisely, for $\varepsilon>0$, we consider the smoothed-out version of the counting form:
\[ \mathcal{N}^{\varepsilon}_{\lambda} :=
\int_{\mathbbm{R}^{2}} \int_{\mathbbm{R}^2} \mathbbm{1}_A(x)\mathbbm{1}_B(x+v) (\sigma_{\lambda}\ast \mathbbm{g}_{\varepsilon\lambda})(v) \,\textup{d}v \,\textup{d}x .\]
We split the counting form into three parts:
\[
\mathcal{N}^{0}_{\lambda} ={ I_{\textup{str}} + I_{\textup{err}} + I_{\textup{uni}}},
\]
where
\[{I_{\textup{str}} := \mathcal{N}_{\lambda
}^{1}}, \quad {I_{\textup{err}}} :=  \mathcal{N}^{\varepsilon}_{\lambda}-\mathcal{N}^{1}_{\lambda}, \quad {I_{\textup{uni}}} :=  \mathcal{N}^{0}_{\lambda}-\mathcal{N}^{\varepsilon}_{\lambda} \]
and call them by their established literature names as, ``the structured part'', ``the error part'' and the ``uniform part'' of $\mathcal{N}^{0}_{\lambda}
$ respectively.
We begin by bounding the structured part from below. We start by recalling an elementary convolution inequality
\begin{equation}\label{eq:elemconvineq}
\sigma_\lambda \ast \mathbbm{g}_{\lambda} \gtrsim  \lambda^{-2}\mathbbm{1}_{[-\lambda,\lambda]^2},
\end{equation}
where the implied constant has the value of $C_{\textup{str}}=\inf_{z \in B(0,3)} \mathbbm{g}(z)$. Using this bound, we immediately obtain that
\begin{align*}
\mathcal{N}^{1}_{\lambda}&\gtrsim {\lambda^{-2}}
\int_{z+[0,R]^2}\int_{\mathbbm{R}^2}\mathbbm{1}_A(x)\mathbbm{1}_B(x+y) \mathbbm{1}_{[-\lambda,\lambda]^2}(y) \,\textup{d}y \,\textup{d}x
\end{align*}
and since $\lambda \in [M,2^{-1}R]$, from \eqref{eq: SzekelyAssumption}, we can conclude
\begin{equation}\label{eq: TwoSetSzekelyStructuredBound}
\mathcal{N}_{\lambda}^1 \geq C_{\textup{str}}\delta R^2.    
\end{equation}
\bigskip
Our next goal is to bound $\lvert I_{\textup{uni}} \rvert$ from above. A standard dominated convergence argument gives that
\[
\lim_{\theta \to 0} \mathcal{N}_{\lambda}^{\theta}=\mathcal{N}_{\lambda}^{0},
\]
so, we actually focus on bounding
\[
\lvert   \mathcal{N}_{\lambda}^{\varepsilon} - \mathcal{N}_{\lambda}^{\theta} \rvert,
\]
for $0<\theta<\varepsilon$. Using the {reparametrized} heat equation \eqref{eq:heateq} and the fundamental theorem of calculus, we obtain
\begin{align*}
\lvert   \mathcal{N}_{\lambda}^{\varepsilon} - \mathcal{N}_{\lambda}^{\theta} \rvert & \lesssim \int_{\theta}^{\varepsilon}\int_{\mathbbm{R}^2}\int_{\mathbbm{R}^2}\mathbbm{1}_A(x)\mathbbm{1}_B(x+y) (\sigma_{\lambda} \ast \mathbbm{k}_{t \lambda})(y)  
\,\textup{d}y \,\textup{d}x \,\frac{\textup{d}t}{t}
\end{align*}
By exploiting the central symmetry of the circle, we can actually recognize that in the integral above, we have the triple convolution
$
\mathbbm{1}_B \ast\sigma_{\lambda} \ast \mathbbm{k}_{t \lambda}
$
and so, by passing to the Fourier side, using the standard formulae, the fact that $A$ and $B$ are subsets of the cube of side-length $R$ and Lemma \ref{lem: Fourier dimension estimate}\,(b), we get
\begin{align*}
    \lvert   \mathcal{N}_{\lambda}^{\varepsilon} - \mathcal{N}_{\lambda}^{\theta} \rvert & \lesssim\int_{\theta}^{\varepsilon} \int_{\mathbbm{R}^2} \lvert \widehat{\mathbbm{1}_A}(\xi)\rvert\lvert\widehat{\mathbbm{1}_B}(\xi)\rvert \lvert \widehat{\sigma_{\lambda}}(\xi) \rvert \lvert \widehat{\mathbbm{k}_{t \lambda}}(\xi)\rvert \, \textup{d}\xi \,\frac{\textup{d}t}{t} \\
    &\lesssim R^2 \int_{\theta}^{\varepsilon} t^{\frac{-1}{2}} \textup{d}t=R^2(\varepsilon^{\frac{1}{2}} - \theta^{\frac{1}{2}}).
\end{align*}
Taking the limit as $\theta$ goes to $0$, we conclude that
\begin{equation}\label{eq: TwoSetSZekelzUniformBound}
\lvert   \mathcal{N}_{\lambda}^{\varepsilon} - \mathcal{N}_{\lambda}^{0} \rvert \leq C_{\textup{uni}}R^{2}\varepsilon^{\frac{1}{2}}    {.}
\end{equation}
\bigskip
Similarly to what we have done for the uniform part, we conclude 
\begin{align*}
\sum_{j=1}^{J}  \big\lvert \mathcal{N}^{1}_{\lambda_j}-\mathcal{N}^{\varepsilon}_{\lambda_j}\big\rvert 
&\lesssim \sum_{j=1}^{J} \int_{\varepsilon}^1 \int_{\mathbbm{R}^2} \lvert\widehat{\mathbbm{1}_A}(\xi)\rvert \lvert \widehat{\mathbbm{1}_B}(\xi) \rvert\lvert \widehat{\sigma_{\lambda_j}}(\xi) \rvert \lvert \widehat{\mathbbm{k}_{t \lambda_j}}(\xi)\rvert \, \textup{d}\xi \,\frac{\textup{d}t}{t} \\
&\lesssim \int_{\mathbbm{R}^2} \int_{\varepsilon}^1 \lvert \widehat{\mathbbm{1}_A}(\xi)\rvert \lvert\widehat{\mathbbm{1}_B}(\xi)\rvert \sum_{j=1}^J \lvert\widehat{\mathbbm{k}}(\lambda_j(t\xi))\rvert \frac{\textup{d}t}{t} \textup{d}\xi
\lesssim \log(\varepsilon^{-1})R^2,
\end{align*}
where the last inequality follows from Lemma \ref{lem: Smooth disjointness estimate} applied with ${\zeta_j} = \lambda_j t \xi$, followed by integration in $t$ and then an application of the Cauchy--Schwarz inequality and Plancherel's theorem. A standard pigeonholing argument now allows us to conclude that there exists a scale $\lambda=\lambda_i \in \{\lambda_1,\lambda_2,\dots,\lambda_J\}$ such that
\begin{equation}\label{eq: TwoSetSzekelyErrorBound}
\big\vert \mathcal{N}^{\varepsilon}_{\lambda_i}(\mathbbm{1}_B)-\mathcal{N}^{1}_{\lambda_i}(\mathbbm{1}_B)\big\vert \leq {C_{\textup{err}}} J^{-1}\log(\varepsilon^{-1})R^2.    
\end{equation}
Combining the estimates \eqref{eq: TwoSetSzekelyStructuredBound}, \eqref{eq: TwoSetSzekelyErrorBound}, \eqref{eq: TwoSetSZekelzUniformBound} together with the choice of parameters \eqref{eq: TwoSetSzekelyParameterChoices} and $\lambda=\lambda_i$, as well as our decomposition of $\mathcal{N}_{\lambda}^0$ into the structured, error and uniform parts, we conclude
\begin{align*}
    \mathcal{N}_{\lambda}^0 & \geq I_{\textup{str}} - \lvert I_{\textup{err}} \rvert - \lvert I_{\textup{uni}}\rvert 
    \geq{C_{\textup{str}}}\Big( \delta R^2 - \frac{\delta R^2}{3} - \frac{\delta R^2}{3}\Big)=  C_{\textup{str}}\frac{\delta R^2}{3}.\qedhere
\end{align*}
\end{proof}
\section{Proof of Theorem \ref{thm: VC dimension, main theorem}}\label{sec: VC dimensions proofs}
We begin this section by proving the aforementioned reduction of Theorem \ref{thm: VC dimension, main theorem} to {Theorem} \ref{thm: VC dimension, configuration variant qualitative}. We then prove Theorem \ref{thm: VC dimension, configuration variant qualitative} employing techniques analogous to the ones used in the proof of Theorem \ref{thm: two set Szekely qualitative}. In particular, we first reduce Theorem \ref{thm: VC dimension, configuration variant qualitative} to a ``quantitative'' version, namely Theorem \ref{thm: VC dimension, configuration variant quantitative} below, which we prove by decomposing a suitable counting form into its structured, error and uniform parts, which we bound separately and devote a subsection to.
\begin{proof}[Theorem \ref{thm: VC dimension, configuration variant qualitative} implies Theorem \ref{thm: VC dimension, main theorem}]
Assume that $A,B \subseteq \mathbbm{R}^2$ are measurable sets for which it holds that $\overline{\delta}_{VC}(A,B)>0$. Applying Theorem \ref{thm: VC dimension, configuration variant qualitative}, we find a critical scale $t_0>0$ such that for an arbitrary $t \geq t_0$, we can find $x \in {B}$ and $v_1,v_2,v_3,s_1,s_2,s_3\in t\Gamma$ which satisfy conditions \eqref{eq: graph points which are translations} and \eqref{eq: graph points which shatter}. We claim that the set
\[
C:=\{x+v_1,x+v_2,x+v_3\}
\]
is shattered by the family
\[
\mathcal{T}_t(A,B)=\{(b+ t\Gamma) \cap A : b\in B\}.
\]
Indeed, first note that $C \subseteq A$, which is certainly a necessary condition for our claim. Further, we check that for every $S \in \mathcal{P}(C)$, we can find $b \in B$ such that both the inclusion
\begin{equation}\label{eq: VCMainInclusion}
S \subseteq C \cap ((b+t\Gamma) \cap A)
\end{equation}
and its reverse hold. This will imply that $VC(\mathcal{T}_t(A,B))\geq 3$. We check this case by case, allowing ourselves the flexibility to redefine $b$ as the context requires it. If $S$ is a $3$-element set, i.e., $S=C$, note that for $j\in\{1,2,3\}$, we have
$(x+v_j)-x=v_j\in t\Gamma$ and so, by taking $b=x$, it follows that \eqref{eq: VCMainInclusion} holds, while in this case the reverse inclusion holds trivially. 
Next we consider the $2$-element case, i.e., we assume that $S=\{x+v_i, x+v_j\}$ for some mutually distinct $i,j \in \{1,2,3\}$. Since
\begin{align*}
    (x+v_i)-(x+v_i+v_j) &= -v_j \in t\Gamma, \qquad
    (x+v_j)-(x+v_i+v_j) = -v_i \in t\Gamma,
\end{align*}
because $\Gamma$ is centrally symmetric, it follows that \eqref{eq: VCMainInclusion} holds with $b=x+v_i+v_j$. To check the reverse inclusion, note that if it were to fail, then the curve segment $(b+t\Gamma) \cap A$ would also have to contain the point $x+v_k$, for $k \in \{1,2,3\} \setminus \{i,j\}$ and so we have that all $3$ points of $C$ lie on the curve $x+t\Gamma$. However, since by Theorem \ref{thm: VC dimension, configuration variant qualitative}, we know that $x$ and $b$ are mutually distinct, we have arrived at a contradiction with Lemma \ref{lem: topological translations along curves} (a). Lastly, we consider the $1$-element case, i.e., we assume that $S=\{x + v_j\}$ for some $j\in \{1,2,3\}$. Because $(x+v_j)-(x+v_j+s_j)=-s_j \in t\Gamma$ and because $t\Gamma$ is centrally symmetric, it follows that \eqref{eq: VCMainInclusion} holds with $b=x+v_j+s_j $.
To see that the reverse inclusion also holds, we again suppose otherwise, i.e., we suppose that for some $i \in \{1,2,3\} \setminus\{j\}$, we have that ${x+v_i} \in (b+t\Gamma) \cap A$. However, by everything we have shown so far, we can conclude that
\[
x+v_i,x+v_j \in x+t\Gamma{\,\cap\,}(x+v_i+v_j)+t\Gamma{\,\cap\,}(x+v_j+s_j)+t\Gamma.
\]
Since, by Theorem \ref{thm: VC dimension, configuration variant qualitative}, the points $x,x+v_i+v_j,x+v_j+s_j$ are mutually distinct, we obtain a contradiction with Lemma \ref{lem: topological translations along curves}\,(b). It remains to show that $VC(\mathcal{T}_t(A,B))= 3$, which we do by showing that no $4$-element set can { be shattered by} $\mathcal{T}_t(A,B)$. Assume otherwise and let 
\[
C=\{c_1,c_2,c_3,c_4\}\subseteq A
\]
be a $4$-element set which {is shattered by} $\mathcal{T}_t(A,B)$. However, by definition, this means that there exist mutually distinct $b_1,b_2,b_3 \in B$ such that
\begin{align*}
    ((b_1+t\Gamma) \cap A)&=\{c_1,c_2,c_3,c_4\}, \;
    ((b_2+t\Gamma) \cap A)=\{c_1,c_2,c_3\}, \;
    ((b_3+t\Gamma) \cap A)=\{c_1,c_2,c_4\},
\end{align*}
but this contradicts Lemma \ref{lem: topological translations along curves} (b), because $c_1$ and $c_2$ both lie on $3$ distinct translates of the curve $t\Gamma$.
\end{proof}
 \begin{theorem}\label{thm: VC dimension, configuration variant quantitative}
     {Let $\delta \in (0,1)$ be arbitrary. Then, there exists a large positive integer $J=J(\delta)$ such that for all $\lambda_1,\lambda_2,\dots,\lambda_J >1$, $R>0$ and $M\geq 1$ which satisfy 
     \begin{itemize}
         \item $2\lambda_k \leq \lambda_{k+1}$ for all $k=1,2,\dots,J-1$,
         \item $R \geq 2\lambda_J$,
         \item $\lambda_1,\lambda_2,\dots,\lambda_J \in [M,2^{-6}R]$,
     \end{itemize}   the following holds:
    For an arbitrary $z\in \mathbbm{R}^2$ and for all $A,B \subseteq z+[0,R]^2$ which satisfy 
    \begin{equation}\label{eq: VC dimension, configuration variant quantitative assumption}
        \inf_{M \leq r \leq 2^{-6}R}
    \frac{1}{(2r)^{12}}\int_{z+[0,R]^2}\int_{([-r,r]^2)^6} \mathcal{F}_{A,B}(x;v_1,v_2,v_3;s_1,s_2,s_3)\,\textup{d}v_1\,\textup{d}v_2\,\textup{d}v_3\,\textup{d}s_1\,\textup{d}s_2\,\textup{d}s_3 \,\textup{d}x\gtrsim \delta R^2
    \end{equation}
    there exists $j \in \{1,2,\dots,J\}$ such that for $\lambda=\lambda_j$ it follows that
    \begin{equation} \label{eq: VC dimension, configuration variant quantitative conclusion}
\int_{(\mathbbm{R}^2)^7}\mathcal{F}_{A,B}(x;v_1,v_2,v_3;s_1,s_2,s_3)
    \,\textup{d}\mu_{\lambda}(v_1)\textup{d}\mu_{\lambda}(v_2)\textup{d}\mu_{\lambda}(v_3)\textup{d}\mu_{\lambda}(s_1)\textup{d}\mu_{\lambda}(s_2)\textup{d}\mu_{\lambda}(s_3) \textup{d}x\gtrsim_{\delta} R^2,
    \end{equation}
    where $\mu$ is as before, the {normalized} (uniform) arclength measure supported on a smooth, closed, centrally symmetric planar curve $\Gamma$, of non-vanishing curvature. Furthermore $\Gamma$ is the boundary of a convex and compact region with non-empty interior.}
 \end{theorem}

\begin{proof}[Theorem \ref{thm: VC dimension, configuration variant quantitative} implies Theorem \ref{thm: VC dimension, configuration variant qualitative}]
 Assume that Theorem \ref{thm: VC dimension, configuration variant quantitative} is true and that Theorem \ref{thm: VC dimension, configuration variant qualitative} is not. Then, there exist measurable sets $A, B \subseteq \mathbbm{R}^2$ such that $\overline{\delta}_{VC}(A,B)>0$ and there exists a sequence of scales $(\lambda_m)_{m \in \mathbbm{N}}$ such that for all $m \in \mathbbm{N}$ and for all $x \in B$ and all $v_1,v_2,v_3,s_1,s_2,s_3 \in \lambda_m\Gamma$ either \eqref{eq: graph points which are translations} or \eqref{eq: graph points which shatter} does not hold.
 By passing to a subsequence, we can assume that $\lambda_1 > 1$ and $2\lambda_m \leq \lambda_{m+1}$ for all $m \in \mathbbm{N}$. We now apply Theorem \ref{thm: VC dimension, configuration variant quantitative} with $\delta=\frac{\overline{\delta}_{VC}(A,B)}{4}$ and let $J$ be as in the conclusion of Theorem \ref{thm: VC dimension, configuration variant quantitative}.
Unpacking Definition \ref{def: density nearness condition}, we first find $M \geq 1$ large enough as well as a sequence of positive real numbers $(R_m)_{m \in \mathbbm{N}}$ converging to infinity and a sequence $(z_m)_{m \in \mathbbm{N}}$ of points in $\mathbbm{R}^2$ such that for all $m \in \mathbbm{N}$
    \begin{align*}
        \inf_{M \leq r \leq 2^{-6}R_m}
    \frac{1}{{(2r)^{12}}}\int_{z_m+[0,R_m]^2}\int_{([-r,r]^2)^6} &\mathcal{F}_{A,B}(x;v_1,v_2,v_3;s_1,s_2,s_3)\\
    &\textup{d}v_1\,\textup{d}v_2\,\textup{d}v_3\,\textup{d}s_1\,\textup{d}s_2\,\textup{d}s_3\,\textup{d}x > \delta R_m^2
    \end{align*}
    Since $(\lambda_k)_{k \in \mathbbm{N}}$ is a lacunary sequence, it follows that there exists some $i \in \mathbbm{N}$ such that $\lambda_i >M$ and since $R_m$ converges to infinity, there exists some $m \in \mathbbm{N}$ such that $\lambda_{i+J-1}<2^{-6}R_m$
    However, Theorem \ref{thm: VC dimension, configuration variant quantitative} now implies that there exists some $j \in \{0,1,\dots,J-1\}$ such that \eqref{eq: VC dimension, configuration variant quantitative conclusion} holds for $\lambda=\lambda_{i+j}$. A standard measure-theoretic argument now provides us with the existence of $x \in B$ and $v_1,v_2,v_3,s_1,s_2,s_3 \in \lambda_{i+j}\Gamma$ which satisfy \eqref{eq: graph points which are translations} and \eqref{eq: graph points which shatter}. Thus, we have arrived at a contradiction.
\end{proof}

Finally, we begin with the proof of Theorem \ref{thm: VC dimension, configuration variant quantitative}. We assume that for $J\in \mathbbm{N}$, which we specify {below}, we are given scales $\lambda_1,\lambda_2,\dots,\lambda_J >1$ which satisfy $2\lambda_k \leq \lambda_{k+1}$ for all $k=1,2,\dots,J-1$, and also $R >0$ such that $R \geq 2\lambda_J$, as well as $M\geq 1$ so that $\lambda_1,\lambda_2,\dots,\lambda_J \in [M,2^{-6}R]$. As before, we introduce the following notation for a given scale $\lambda$
\[ \mathcal{N}^{0}_{\lambda} :=
\int_{(\mathbbm{R}^2)^{7}} \mathcal{F}_{A,B}(x;v_1,v_2,v_3;s_1,s_2,s_3) \,\textup{d}\mu_{\lambda}(v_1)\,\textup{d}\mu_{\lambda}(v_2)\,\textup{d}\mu_{\lambda}(v_3)\,\textup{d}\mu_{\lambda}(s_1)\,\textup{d}\mu_{\lambda}(s_2)\,\textup{d}\mu_{\lambda}(s_3) \,\textup{d}x \]
and, for $\varepsilon>0$, we also consider
\begin{align*}
\mathcal{N}^{\varepsilon}_{\lambda} :=
\int_{(\mathbbm{R}^2)^{7}}  \mathcal{F}_{A,B}(x;v_1,v_2,v_3;s_1,s_2,s_3)&\prod_{i=1}^3(\mu_{\lambda}\ast \mathbbm{g}_{\varepsilon\lambda})(v_i)
\prod_{i=1}^3(\mu_{\lambda}\ast \mathbbm{g}_{\varepsilon\lambda})(s_i) \\ &{\textup{d}v_1\,\textup{d}v_2\,\textup{d}v_3\,\textup{d}s_1\,\textup{d}s_2\,\textup{d}s_3\,\textup{d}x}.
\end{align*}
We split the counting form $\mathcal{N}_{\lambda}^0$ into three parts:
\[
\mathcal{N}^{0}_{\lambda} = I_{\textup{str}} + I_{\textup{err}} + I_{\textup{uni}},
\]
where 
\[{I_{\textup{str}} := \mathcal{N}_{\lambda}^{1}}, \quad I_{\textup{err}} :=  \mathcal{N}^{\varepsilon}_{\lambda}-\mathcal{N}^{1}_{\lambda} , \quad I_{\textup{uni}} := \mathcal{N}^{0}_{\lambda} - \mathcal{N}^{\varepsilon}_{\lambda}.\]
Once more, we call these parts ``the structured part'', ``the error part'' and ``the uniform part'', respectively.
Throughout the proof, we will again use the letter $\lambda$ to mean any one of the scales $\lambda_1,\lambda_2,...,\lambda_J$ whenever the claim in question does not depend on which particular scale we are working with.
\subsection{The Structured Part bound}\label{sec: VC Structured part}
By applying an estimate analogous to \eqref{eq:elemconvineq}, we conclude that
\begin{align*}
\mathcal{N}^{1}_{\lambda} \gtrsim {\lambda^{-12}}
\int_{(\mathbbm{R}^2)^{7}}  &\mathcal{F}_{A,B}(x;v_1,v_2,v_3;s_1,s_2,s_3)\prod_{i=1}^3\mathbbm{1}_{[-\lambda,\lambda]^2}(v_i)
\prod_{i=1}^3\mathbbm{1}_{[-\lambda,\lambda]^2}(s_i) \\ &{\textup{d}v_1\,\textup{d}v_2\,\textup{d}v_3\,\textup{d}s_1\,\textup{d}s_2\,\textup{d}s_3\,\textup{d}x},
\end{align*}
and using assumption \eqref{eq: VC dimension, configuration variant quantitative assumption}, we deduce
\begin{equation}\label{eq: VCStructuredPart}
\mathcal{N}^{1}_{\lambda} \geq C_{\textup{str}}\delta R^2,
\end{equation}
for some constant $C_{\textup{str}}>0$.
\subsection{The Error Part bound}\label{sec: VC error part}
We will show
\begin{equation}\label{eq: VCErrorPart}
\sum_{j=1}^{J}  \big\vert \mathcal{N}^{1}_{\lambda_j}-\mathcal{N}^{\varepsilon}_{\lambda_j}\big\vert \leq C_{\textup{err}} \varepsilon^{-18} \log(\varepsilon^{-1}) R^{2},
\end{equation}
for some constant $C_{\textup{err}}>0$. By using the fundamental theorem of calculus and the reparametrized heat equation \eqref{eq:heateq}, we can write
\[
\mathcal{N}_{\lambda}^{1}-\mathcal{N}_{\lambda}^{\varepsilon} = {\sum_{k=1}^{6} \mathcal{L}_{\lambda}^{k,\varepsilon,1}},
\]
where for a fixed $k \in \left\{ 1,2,\dots,6 \right\}$,
\begin{align*}
\mathcal{L}_{\lambda}^{k,\varepsilon,1} &:= \int_{\varepsilon}^{1}\int_{(\mathbbm{R}^2)^7}\mathcal{F}_{6}(x;v_{1},v_{2},v_{3},v_{4},v_{5},v_{6})\,({\mu_\lambda \ast\mathbbm{k}_{t \lambda}})(v_k)\prod_{\substack{i=1 \\ i\neq k}}^{6}(\mu_{\lambda} \ast\mathbbm{g}_{t \lambda})(v_i) \,\textup{d}v_1\dots \textup{d}v_{6}\, \textup{d}x\frac{\textup{d}t}{t} \\
\end{align*}
and where $\mathcal{F}_{6}(x;v_{1},v_{2},v_{3},v_{4},v_{5},v_{6})$ is as in Definition \ref{def: Joint Upper Banach Density} for the particular choice of functions described in \eqref{eq: VCChoiceOfVertices}, {except that instead of $\mathbbm{1}_{\mathbbm{R}^2}$ we are assigning $\mathbbm{1}_{z+[-R,2R]^2}$ to the vertices which do not get assigned either $\mathbbm{1}_A$ or $\mathbbm{1}_B$. This reflects the localization done in Theorem \ref{thm: VC dimension, configuration variant quantitative} }. Our first goal is to bound away all appearances of the measure $\mu$ in a given $\mathcal{L}_{\lambda}^{k,\varepsilon,1}$. We introduce two additional parameters, which should be understood as auxiliary convolution
scales.
For $e':=10^{-2}e^{-1}$, which is just a suitably chosen small constant, and any $\lambda,t>0$, we have
$$
\int_{e' t \lambda}^{ee't\lambda} \frac{\textup{d}s}{s} = 1,
$$
so we can introduce an additional integral into the form { $\mathcal{L}_{\lambda}^{k,\varepsilon,1}$}, giving us slightly more flexibility. We also denote
\[
r := \sqrt{ (t \lambda)^{2}- 2s^2 }.
\]
and observe that, { provided that $s\in[e't\lambda,ee't\lambda]$, we have } $t\lambda \sim r \sim s.$ Using \eqref{eq: Convolution of partial derivatives of Gaussians} and \eqref{eq: Convolution of a gaussian Laplacian and a gaussian}, we get that 
\[
{\mu_{\lambda}} \ast \mathbbm{k}_{t \lambda} = \frac{(t \lambda)^2}{s^{2}}\sum_{l=1}^{2}\mathbbm{h}_{s}^{(l)}\ast\mathbbm{h}_{s}^{(l)}\ast\mathbbm{g}_{r}\ast{\mu_{\lambda}},
\]
and using this decomposition, followed by an application of \eqref{eq: Domination of Gaussian with t lambda} and \eqref{eq: Domination of Gaussian with r}, we bound $\lvert \mathcal{L}_{\lambda}^{k,\varepsilon,1}\rvert$ by
\begin{align*}
\varepsilon^{-18}\sum_{l=1}^2\int_{[1,\infty)^6}\int_{\varepsilon}^{1}\int_{e't\lambda}^{ee't\lambda}\int_{(\mathbbm{R}^2)^7}\int_{(\mathbbm{R}^2)^2}&\mathcal{F}_{6}(x;v_{1},\dots,v_{6})\mathbbm{h}_{s}^{(l)}(v_k-y_{1})\mathbbm{h}_{s}^{(l)}(y_{1}-y_{2})\\
&\mathbbm{g}_{s\gamma_{k}}(y_{2}) 
\prod_{\substack{i=1 \\ i\neq k}}^6(\mathbbm{g}_{s \gamma_{i}})(v_i) \, \textup{d}y_{1}\,\textup{d}y_{2}\,\textup{d}v_1\dots \textup{d}v_{6}\,\textup{d}x\\
&\frac{\textup{d}s}{s}\,\frac{\textup{d}t}{t}\,\frac{\textup{d}\gamma_{1}}{\gamma_{1}^2}\dots\frac{\textup{d}\gamma_{6}}{\gamma_{6}^2}.
\end{align*}
We can now collapse the convolution back, using \eqref{eq: Convolution of partial derivatives of Gaussians} and \eqref{eq: Convolution of a gaussian Laplacian and a gaussian} once more, in order to get that the above {is bounded by}
\begin{align*}
\varepsilon^{-18}\int_{\varepsilon}^{1}\int_{e't\lambda_{j}}^{ee't\lambda_{j}}\int_{[1,\infty)^6}\int_{(\mathbbm{R}^2)^7}&\mathcal{F}_{6}(x;v_{1},\dots,v_{6})\mathbbm{k}_{s \sqrt{2 + \gamma_{k}^2 }}(v_k) \\
&\prod_{\substack{i=1 \\ i\neq k}}^6(\mathbbm{g}_{s \gamma_{i}})(v_i) \,\textup{d}v_1\dots \textup{d}v_{6}\, \textup{d}x\,\frac{\textup{d}\gamma_{1}}{\gamma_{1}^2}\dots\frac{\textup{d}\gamma_{k}}{\gamma_{k}^2(2+\gamma_{k}^2)}\dots\frac{\textup{d}\gamma_{6}}{\gamma_{6}^2}\,\frac{\textup{d}s}{s}\,\frac{\textup{d}t}{t}.
\end{align*}
We also have to sum over the scales $\lambda_j$. At this point, we use the lacunary nature of the scales in question. For a fixed $t \in (\varepsilon,1)$, each interval $[e't\lambda_{j},ee't\lambda_{j}]$ can overlap at most two intervals of the form $[e't\lambda_k,ee't\lambda_k]$ for $k=1,2,\dots,J$. Therefore, a fixed $s \in [0,\infty)$ can lie only in two such intervals. Thus, after a routine application of Fubini's theorem, we can integrate in the $t$-variable and conclude that $\sum_{j=1}^{J}\lvert \mathcal{L}_{\lambda_{j}}^{k,\varepsilon,1} \rvert$ is bounded by
\begin{align*}
2\varepsilon^{-18}\log(\varepsilon^{-1})\int_{[1,\infty)^6}\int_{0}^{\infty}\int_{(\mathbbm{R}^2)^7}&\mathcal{F}_{6}(x;v_{1},\dots,v_{6})\mathbbm{k}_{s \sqrt{2 + \gamma_{k}^2 }}(v_{k}) \\
&\prod_{\substack{i=1 \\ i\neq k}}^6(\mathbbm{g}_{s \gamma_{i}})(v_i) \,\textup{d}v_1\dots \textup{d}v_{6}\, \textup{d}x\frac{\textup{d}s}{s}\frac{\textup{d}\gamma_{1}}{\gamma_{1}^2}\dots\frac{\textup{d}\gamma_{k}}{\gamma_{k}^2(2+\gamma_{k}^2)}\dots\frac{\textup{d}\gamma_{6}}{\gamma_{6}^2}.
\end{align*}
However, we now recognize that the above is actually a superposition of so-called singular Brascamp--Lieb forms, for which, in the following subsection, we obtain bounds which are uniform in the kernel scales. That is to say, Theorem \ref{thm: Singular Brascam-Lieb bound} below, allows us to conclude
\begin{align*}
\sum_{j=1}^{J}\lvert \mathcal{L}_{\lambda_{j}}^{k,\varepsilon,1} \rvert
&\lesssim \varepsilon^{-18}\log(\varepsilon^{-1}) {\int_{[1,\infty)^6}} \prod_{\eta\in \left\{ 0,1 \right\}^{6}}\lVert f_{\eta} \rVert_{\textup{L}^{2^{6}}} \frac{\textup{d}\gamma_{1}}{\gamma_{1}^2}\dots\frac{\textup{d}\gamma_{k}}{\gamma_{k}^2(2+\gamma_{k}^2)}\dots\frac{\textup{d}\gamma_{6}}{\gamma_{6}^2} \\
&\lesssim \varepsilon^{-18} \log(\varepsilon^{-1}) R^{2},
\end{align*}
which implies the desired bound.

\subsection{A particular case of singular Brascamp--Lieb inequalities}\label{sec: Singular Brascamp-Lieb}
We start by fixing notation which shall be used throughout this subsection.
For a fixed positive integer $n\in \mathbbm{N}$, {$j \in \{1,2,\dots, n\}$} and fixed scales $\alpha_1,\alpha_2, \dots, \alpha_n >0$ we consider the integral kernel $K^{\alpha_1,\alpha_2, \dots, \alpha_n}_j: (\mathbbm{R}^2)^n\to [0,+\infty)$
\[
K^{\alpha_1,\alpha_2, \dots, \alpha_n}_j(v_1,v_2,\dots,v_n) := \mathbbm{k}_{\alpha_j}(v_j) \prod_{\substack{ i = 1 \\ i \neq j}}^n \mathbbm{g}_{\alpha_i}(v_i).
\]
We also consider a finite family of functions $(f_{\eta})_{\eta \in \{0,1\}^{n}}$ in $\textup{L}^{2^n}(\mathbbm{R}^2)$. Our main object of interest is the following multilinear integral form
\begin{align*}
\mathcal{J}_j^{\alpha_{1},\alpha_{2},\dots,\alpha_{n}}((f_{\eta})_{\eta \in \{0,1\}^n}):=- \int_0^{\infty}\int_{(\mathbbm{R}^{2})^{n+1}}
    &\prod_{\eta \in \{0,1\}^n} f_{\eta}(x+\eta\cdot v) \\
    & K^{s\alpha_1,s\alpha_2, \dots, s\alpha_n}_j(v_1,v_2,\dots,v_n)\,\textup{d}v_1\dots \textup{d}v_n \, \textup{d}x\frac{\textup{d} s }{s}.
\end{align*}
{ We note that the minus sign in the definition ensures our forms are positive in the particular case we consider in Lemma \ref{lem: The case of the fixed Laplacian} below. }The main result of this subsection is
\begin{theorem}\label{thm: Singular Brascam-Lieb bound}
For any functions $(f_{\eta})_{\eta \in \{0,1\}^{n}}$ in $\textup{L}^{2^n}(\mathbbm{R}^2)$, any $\alpha_1,\alpha_2, \dots, \alpha_n >0$ and any $j \in \{1,2,\dots, n\}$, it follows that
\[
\lvert \mathcal{J}_j^{\alpha_{1},\alpha_{2},\dots,\alpha_{n}}((f_{\eta})_{\eta \in \{0,1\}^n})\rvert \lesssim_n \prod_{\eta \in \left\{ 0,1 \right\}^{n}}\lVert f_{\eta} \rVert_{\textup{L}^{2^n}(\mathbbm{R}^2)} {.}
\]
\end{theorem}
 The key idea of the proof is to gradually apply the Cauchy--Schwarz inequality, in order to ``symmetrize'' the form under consideration. Note that, a priori, we do not even know that our forms are well defined. However, we will prove the desired bound in the Proposition \ref{prop: Singular Brascamp Lieb inductive statement} by assuming that all functions under consideration belong to the Schwartz class. Theorem \ref{thm: Singular Brascam-Lieb bound} then easily follows by applying classic functional-analytic tools and standard approximation arguments, which both guarantee that our form is well defined and bounded. 

\begin{prop}\label{prop: Singular Brascamp Lieb inductive statement}
For any $n \in \mathbbm{N}$, any functions $(f_\eta)_{\eta \in \{0,1\}^n}$ belonging to the Schwartz class $\mathcal{S}(\mathbbm{R}^2)$ and for any $0\leq k \leq {n-1}$ it follows that for mutually distinct $i_{1},i_{2},\dots,i_{k} \in \left\{ 1,2,\dots,n \right\}$ and for all $a_{1},a_{2},\dots,a_{k} \in \left\{ 0,1 \right\}$ and all $j \in \left\{ 1,2,\dots,n \right\}\setminus \left\{ i_{1},i_{2},\dots,i_{k} \right\}$, the following bound holds
\[
\lvert \mathcal{J}_j^{\alpha_{1},\alpha_{2},\dots,\alpha_{n}}((f_{C_{(i_{1},a_{1}),(i_{2},a_{2}),\dots,(i_{k},a_{k})}(\eta)})_{\eta\in \{0,1\}^n})\rvert \lesssim_{n} \prod_{\eta \in \left\{ 0,1 \right\}^{n}}\lVert f_{C_{(i_{1},a_{1}),(i_{2},a_{2}),\dots,(i_{k},a_{k})}(\eta)} \rVert_{\textup{L}^{2^n}(\mathbbm{R}^{2})}{.}
\]
\end{prop}
{The proof of Proposition \ref{prop: Singular Brascamp Lieb inductive statement} will be done by backward induction on $0\leq k\leq {n-1}$. The base case $k=n-1$ follows the exact same proof as the proof of the inductive step which we show below. The only difference is that in the base case, instead of using the inductive hypothesis, we use the fact that the claim of Theorem \ref{thm: Singular Brascam-Lieb bound} holds in the case of all function $f_{\eta}$ being mutually equal. This fact is established by combining Lemmas \ref{lem: The case of the fixed Laplacian} and \ref{lem: continuous telescoping lemma} below, in a way similar to how the estimate \eqref{eq: UsingPositivitySBL} below is established.} Notice that Theorem \ref{thm: Singular Brascam-Lieb bound} is precisely recovered as the case $k=0$, where we agree that the cardinality of the empty set equals $0$, modulo the aforementioned functional analysis density argument. {We stress the importance of the assumption that the coordinate which carries the Laplacian has not appeared among the coordinates which have been fixed by the freeze function. This is precisely what will meaningfully allow us to reduce the number of functions we work with at a given step.} However, we shall also need to know what happens in the complementary case, i.e., if the ``Laplacian coordinate'' is actually fixed by the freeze function. {This is done in Lemma \ref{lem: The case of the fixed Laplacian} below, a variant of which has been used in a similar context in \cite{DK21}, \cite{DK22}, \cite{K20:anisotrop} and \cite{KP2023}. The proof of the lemma follows the approaches from \cite{Kov12} and \cite{Dur15}}. In order to be able to connect the two cases, we shall need Lemma \ref{lem: continuous telescoping lemma} below, which we state without proof because it is formally identical to the proof of the analogous statement in \cite{KP2023}. We also advise the interested reader to look at Lemma 3 in \cite{Dur15} and Lemma 4 in \cite{DKR18}.  
\begin{lem}\label{lem: The case of the fixed Laplacian}
For fixed scales $\alpha_1,\alpha_2, \dots, \alpha_n >0$, for functions $(f_{\eta})_{\eta \in \{0,1\}^{n}}$ in $\textup{L}^{2^n}(\mathbbm{R}^2)$, for fixed $k \leq n$, any coordinates $i_{1},i_{2},\dots,i_{k} \in \left\{ 1,2,\dots,n \right\}$ and any $a_1,a_2,\dots,a_k \in \{0,1\}$, we claim that for each ${j} \in \left\{ i_{1},i_{2},\dots,i_{k} \right\}$, we have that
\[
\mathcal{J}^{\alpha_{1},\alpha_{2},\dots,\alpha_{n}}_{{j}}((f_{C_{(i_{1},a_{1}),(i_{2},a_{2}),\dots,(i_{k},a_{k})}(\eta)})_{\eta \in \left\{ 0,1 \right\}^{n}}) \geq 0.
\]
\end{lem}
\begin{proof}
{
For a tuple of vectors $v\in (\mathbbm{R}^2)^n$, we denote 
\[
\mathcal{F}_v^{n-1}(y):=\prod_{\substack{\eta\in \left\{ 0,1 \right\}^{n}\\\eta_j=0}}f_{C_{(i_{1},a_{1}),(i_{2},a_{2}),\dots,(i_{k},a_{k})}(\eta)}(y+\eta \cdot v).
\]
}
Using \eqref{eq: Convolution of partial derivatives of Gaussians} combined with the fact that a convolution of two odd functions is even, after a simple calculation, we obtain that for some $C>0$, we have 
\begin{align*}
{\mathcal{J}^{\alpha_{1},\alpha_{2},\dots,\alpha_{n}}_j}=C \sum_{l=1}^{2}\int_{0}^{\infty} \int_{(\mathbbm{R}^2)^{n-1}}&\left\lVert{\mathcal{F}_v^{n-1}}\ast\mathbbm{h}_{\frac{s\alpha_{j}}{\sqrt{ 2 }}}^{(l)} \right\rVert_{\textup{L}^2(\mathbbm{R}^2)}^2 \prod_{\substack{i=1\\i\neq j}}^n\mathbbm{g}_{s\alpha_{i}}(v_{i})\,\prod_{\substack{i=1\\i\neq j}}^{n}\textup{d}v_i\,\frac{\textup{d} s }{s}
\end{align*}
which implies the desired claim.
\end{proof}
\begin{lem}\label{lem: continuous telescoping lemma}
   For a finite family of functions $(f_{\eta})_{\eta \in \{0,1\}^{n}}$ in $\mathcal{S}(\mathbbm{R}^2)$, and $\alpha_1,\alpha_2, \dots, \alpha_n >0$, it holds that
    \[
    \sum_{i=1}^n \mathcal{J}_i = 2 \pi \int_{\mathbbm{R}^2}\prod_{\eta\in \{0,1\}^n} f_{\eta}(x) \, \textup{d}x,
    \]
where we have briefly denoted $\mathcal{J}_i = \mathcal{J}_i^{\alpha_1,\alpha_2,\dots,\alpha_n}((f_\eta)_{\eta \in \{0,1\}^n})$.
\end{lem}
\begin{proof}[Proof of Proposition \ref{prop: Singular Brascamp Lieb inductive statement}]
As has already been mentioned the proof proceeds by induction on $k$. We show how to make an inductive step, i.e., how we reduce a number of different functions appearing in our form by a factor of $2$.
Let $k<n$ be arbitrary and also let $j \in \left\{ 1,2,\dots,n \right\}$, $i_{1},i_{2},\dots,i_{k} \in \left\{ 1,2,\dots,n \right\}\setminus\{j\}$ and $a_{1},a_{2},\dots,a_{k} \in \left\{ 0,1 \right\}$ be arbitrary. By using \eqref{eq: Convolution of partial derivatives of Gaussians}, introducing the change of variables $y=x+v_{j}$, $\textup{d}y=\textup{d}v_{j}$ and expanding the convolution,
we conclude that $\mathcal{J}_{j}^{\alpha_{1},\alpha_{2},\dots,\alpha_{n}}((f_{C_{(i_{1},a_{1}),(i_{2},a_{2}),\dots,(i_{k},a_{k})}(\eta)})_{\eta\in \{0,1\}^n})$ is, up to a constant, equal to
\begin{align*}
\sum_{l=1}^{2}\int_{0}^{\infty}\int_{(\mathbbm{R}^2)^{n}}&\bigg(\int_{\mathbbm{R}^2}\prod_{\substack{\eta\in \left\{ 0,1 \right\}^{n}\\\eta_{j}=0}}f_{C_{(i_{1},a_{1}),(i_{2},a_{2}),\dots,(i_{k},a_{k})}(\eta)}(x+\eta \cdot v)\mathbbm{h}_{\frac{s\alpha_{j}}{\sqrt{ 2 }}}^{(l)}(u-x)\,\textup{d}x\bigg)\\
&\bigg(\int_{\mathbbm{R}^2}\prod_{\substack{\eta\in \left\{ 0,1 \right\}^{n}\\\eta_{j}=1}}f_{C_{(i_{1},a_{1}),(i_{2},a_{2}),\dots,(i_{k},a_{k})}(\eta)}(y+C_{(j,0)}(\eta) \cdot v)\mathbbm{h}_{\frac{s\alpha_{j}}{\sqrt{2}}}^{(l)}(y-u)\,\textup{d}y\bigg)\\
&\prod_{\substack{i=1\\i\neq j}}^n\mathbbm{g}_{s\alpha_{i}}(v_{i})\,\textup{d}u\,\prod_{\substack{i=1\\i\neq j}}^n\textup{d}v_i\frac{\textup{d} s }{s}
\end{align*}
By applying the Cauchy--Schwarz inequality, expanding the square in each of the two factors, making changes of variables analogous to the one already made, exploiting the fact that $\mathbbm{h}^{(l)}$ is odd and finally once more collapsing the two convolutions in each of the two factors by again invoking \eqref{eq: Convolution of partial derivatives of Gaussians}, we conclude that we can bound 
\[\lvert \mathcal{J}_{j}^{\alpha_{1},\alpha_{2},\dots,\alpha_{n}}((f_{C_{(i_{1},a_{1}),(i_{2},a_{2}),\dots,(i_{k},a_{k})}(\eta)})_{\eta\in \{0,1\}^n}) \rvert^2\] 
by the product
\begin{equation}\label{eq: inductive step bound}
\begin{aligned}
&\lvert \mathcal{J}_{j}^{\alpha_{1},\alpha_{2},\dots,\alpha_{n}}((f_{C_{(i_{1},a_{1}),(i_{2},a_{2}),\dots,(i_{k},a_{k}),(j,0)}(\eta)})_{\eta\in \{0,1\}^n})\rvert\\\cdot 
&\lvert \mathcal{J}_{j}^{\alpha_{1},\alpha_{2},\dots,\alpha_{n}}((f_{C_{(i_{1},a_{1}),(i_{2},a_{2}),\dots,(i_{k},a_{k}),(j,1)}(\eta)})_{\eta\in \{0,1\}^n})\rvert
\end{aligned}
\end{equation}
Note that we cannot immediately apply the inductive hypothesis on 
\[
\mathcal{J}_{j}^{\alpha_{1},\alpha_{2},\dots,\alpha_{n}}((f_{C_{(i_{1},a_{1}),(i_{2},a_{2}),\dots,(i_{k},a_{k}),(j,a_j)}(\eta)})_{\eta\in \{0,1\}^n}),
\]
for $a_j \in \{0,1\}$, because the coordinate $j$ appears as one of the $k+1$ fixed coordinates in our freeze function and we explicitly excluded such cases in our inductive statement. Still, we will circumvent this problem by using Lemma \ref{lem: continuous telescoping lemma}. Let $(i_{k+1},a_{k+1})$ stand for either one of the two pairs $(j,0)$ or $(j,1)$ above.
Denote the following
\begin{align*}
\mathcal{S}_{1}&:=\sum_{\substack{l=1 \\ l\in \left\{ i_{1},i_{2},\dots,i_{k+1} \right\}}}^n\mathcal{J}_{l}^{\alpha_{1},\alpha_{2},\dots,\alpha_{n}}((f_{C_{(i_{1},a_{1}),(i_{2},a_{2}),\dots,(i_{k+1},a_{k+1})}(\eta)})_{\eta\in \{0,1\}^n}), \\
\mathcal{S}_{2}&:=\sum_{\substack{l=1 \\ l\not\in \left\{ i_{1},i_{2},\dots,i_{k+1} \right\}}}^n\mathcal{J}_{l}^{\alpha_{1},\alpha_{2},\dots,\alpha_{n}}((f_{C_{(i_{1},a_{1}),(i_{2},a_{2}),\dots,(i_{k+1},a_{k+1})}(\eta)})_{\eta\in \{0,1\}^n}).
\end{align*}
We can bound $\lvert \mathcal{S}_2\rvert$ by using the inductive hypothesis and we can also bound $\lvert \mathcal{S}_{1} +\mathcal{S}_{2} \rvert$ by using Lemma \ref{lem: continuous telescoping lemma} and H\"older's inequality. Therefore, we obtain a bound for $\lvert \mathcal{S}_1\rvert$
\begin{align*}
 \lvert \mathcal{S}_{1} \rvert&\leq \lvert \mathcal{S}_{1} + \mathcal{S}_{2} \rvert+\lvert \mathcal{S}_{2} \rvert    \\
&\leq(2\pi+(n-k))
\prod_{\eta \in \left\{ 0,1 \right\}^{n}}\lVert f_{C_{(i_{1},a_{1}),(i_{2},a_{2}),\dots,(i_{k+1},a_{k+1})}(\eta)} \rVert_{\textup{L}^{{2^n}}(\mathbbm{R}^{2})}.
\end{align*}
In particular, using positivity guaranteed to us by Lemma \ref{lem: The case of the fixed Laplacian}, we conclude that 
\begin{equation}\label{eq: UsingPositivitySBL}
\begin{aligned}
\lvert &\mathcal{J}^{\alpha_{1},\alpha_{2},\dots,\alpha_{n}}_{j}((f_{C_{(i_{1},a_{1}),(i_{2},a_{2}),\dots,(i_{k+1},a_{k+1})}(\eta)})_{\eta \in \left\{ 0,1 \right\}^{n}})\rvert\\ &\leq \lvert \mathcal{S}_1 \rvert 
\lesssim_n\prod_{\eta \in \left\{ 0,1 \right\}^{n}}\lVert f_{C_{(i_{1},a_{1}),(i_{2},a_{2}),\dots,(i_{k+1},a_{k+1})}(\eta)} \rVert_{\textup{L}^{2^n}(\mathbbm{R}^{2})},
\end{aligned}
\end{equation}
which, after a brief calculation, gives us that \eqref{eq: inductive step bound} is bounded by
\begin{align*}
\prod_{\eta \in \left\{ 0,1 \right\}^n}\lVert f_{C_{(i_{1},a_{1}),(i_{2},a_{2}),\dots,(i_{k},a_{k})}(\eta)} \rVert_{\textup{L}^{2^n}(\mathbbm{R}^{2})}^2.
\end{align*}
By taking square roots, we obtain the desired claim.

\end{proof}
\subsection{The Uniform Part bound}\label{sec: VC uniform part}
For $0<\theta < \varepsilon < 1$, we focus on bounding
\[
\lvert \mathcal{N}_{\lambda}^{\theta} - \mathcal{N}_{\lambda}^{\varepsilon} \rvert.
\]
Similar to before, using the fundamental theorem of calculus and the reparametrized heat equation, we can write
\[
\mathcal{N}_{\lambda}^{\varepsilon}-\mathcal{N}_{\lambda}^{\theta} = \sum_{{k}=1}^{6} \mathcal{L}_{\lambda}^{k,\theta,\varepsilon},
\]
where for a fixed $k \in \left\{ 1,2,\dots,6 \right\}$, we have denoted, similarly as before,
\begin{align*}
\mathcal{L}_{\lambda}^{k,\theta,\varepsilon} &:= {\int_{\theta}^{\varepsilon}}\int_{(\mathbbm{R}^2)^7}\mathcal{F}_{6}(x;v_{1},v_{2},v_{3},v_{4},v_{5},v_{6})\,({\mu_\lambda \ast\mathbbm{k}_{t \lambda}})(v_k)\prod_{\substack{i=1 \\ i\neq k}}^{6}(\mu_{\lambda} \ast\mathbbm{g}_{t \lambda})(v_i) \,\textup{d}v_1\dots \textup{d}v_{6}\, \textup{d}x\frac{\textup{d}t}{t}. \\
\end{align*}
For fixed variables $v_1,\dots,v_{k-1},v_{k+1},\dots,v_6$, we denote
\[
F_k(u):=\prod_{\substack{\eta\in\{0,1\}^6\\ \eta_k=0}}\mathbbm{1}_{z+[0,R]^2}(x+\eta_1v_1+\dots +\eta_{k-1}v_{k-1}+{u}+\eta_{k+1}v_{k+1} + \dots +\eta_6 v_6).
\]
By noticing that ${\mathcal{F}_6}(x;v_1,\dots,v_6)\leq F_k(x+{v_k})F_k(x)$ and by making a change of variables $u=x+v_k$, $\textup{d}u = \textup{d}v_k$, we obtain the bound
\[
\mathcal{L}_{\lambda}^{k,\theta,\varepsilon} \leq {\int_{\theta}^{\varepsilon}}\int_{(\mathbbm{R}^2)^5}\prod_{\substack{i=1 \\ i\neq k}}^{6}(\mu_{\lambda} \ast\mathbbm{g}_{t \lambda})(v_i)\int_{\mathbbm{R}^2}(F_k \ast \mu_{\lambda}\ast \mathbbm{k}_{t \lambda})(u)F_k(u) \,\textup{d}u\,\textup{d}v_1\dots\textup{d}v_{k-1}\textup{d}v_k\dots \textup{d}v_6\,\frac{\textup{d}t}{t}.
\]
By applying Plancherel's formula in the inner integral, followed by Lemma \ref{lem: Fourier dimension estimate}\,(b) and then integrating in all the remaining variables, we conclude that
\[
\lvert \mathcal{L}_{\lambda}^{k,\theta,\varepsilon} \rvert \lesssim R^2\int_{\theta}^{\varepsilon} t^\frac{-1}{2}\,\textup{d}t,
\]
By integrating in $t$ and using the dominated convergence theorem to let $\theta  \to 0$, we obtain the bound 
\begin{equation}\label{eq: VCUniform}
\lvert \mathcal{N}_{\lambda}^{\varepsilon}-\mathcal{N}_{\lambda}^{0}  \rvert \leq C_{\textup{uni}} \varepsilon^\frac{1}{2}R^{2},
\end{equation}
for some constant $C_{\textup{uni}}>0$.
\subsection{Combining the estimates}\label{sec: VC estimate combination}
Finally, we choose
\begin{equation}\label{eq: VCParameterChoice}
{\varepsilon=\frac{1}{2}\Big(\frac{\delta C_{\textup{str}}}{3C_{\textup{uni}}}\Big)^2,\qquad J=\left\lceil\frac{3C_{\textup{err}}\log(\varepsilon^{-1})}{\delta \varepsilon^{18}C_{\textup{str}}}\right\rceil } 
\end{equation}
From \eqref{eq: VCErrorPart}, by pigeonholing in $j$, we can find a ``good scale'' $\lambda_j$ such that
\[
\lvert I_{\textup{err}} \rvert\leq \frac{C_{\textup{str}}}{3} \delta R^2.
\]
Combining this with \eqref{eq: VCStructuredPart} and \eqref{eq: VCUniform}, we obtain the desired claim
\[
\mathcal{N}^{0}_{\lambda_j}\geq I_{\textup{str}} - \lvert I_{\textup{err}}\rvert - \lvert I_{\textup{uni}}\rvert\geq\frac{C_{\textup{str}}}{3} \delta R^2.
\]
{\appendix
\section{An alternative formulation of upper Banach density}
The goal of this appendix is to show that the limit superior appearing in Definition \ref{def: upper Banach density} can be replaced by a limit. Throughout this appendix, we fix an integer $d \in \mathbbm{N}$ and fix a norm $\lVert \cdot \Vert$ on $\mathbbm{R}^d$. For $r>0$, let $B_r$ denote the ball with respect to $\lVert \cdot \rVert$, centered at the origin. The following two lemmas are essentially identical to calculations appearing as part of the proof of Lemma 3.1 from \cite{FKY22}. The only difference is that the authors of \cite{FKY22} consider only the case of the $\ell^2$ norm on $\mathbbm{R}^d$.
\begin{lem}\label{lem: FKY 1}
For every $r>s>0$, $z \in \mathbbm{R}^d$ and every measurable set $A\subseteq \mathbbm{R}^d$, it holds that
        \[
        \int_{z + B_r} \left\vert A \cap\left(y+ B_s\right)\right\vert \textup{d}y \geq
        \left\vert A \cap\left(z+ B_{r-s}\right)\right\vert  \left\vert B_{s}\right\vert.
        \]
\end{lem}
\begin{proof}
    Using Fubini's theorem and elementary set inclusions, we can calculate \begin{align*}
            \int_{z + B_r} \left\vert A \cap\left(y+ B_s\right)\right\vert \textup{d}y 
            &\geq \int_{\mathbbm{R}^d}\int_{\mathbbm{R}^d} \mathbbm{1}_{z + B_{r-s}}(x)\mathbbm{1}_{A}(x)\mathbbm{1}_{x+ B_s}(y) \textup{d}x\textup{d}y \\
            &=\left\vert A \cap\left(z+ B_{r-s}\right)\right\vert  \left\vert B_{s}\right\vert.
        \end{align*}
\end{proof}
\begin{lem}\label{lem: FKY 2}
For every $r>s>0$, $z \in \mathbbm{R}^d$ and every measurable set $A\subseteq \mathbbm{R}^d$, it holds that
  \[
    \left\vert A \cap\left(z+ B_{r-s}\right)\right\vert
    \geq \left\vert A \cap\left(z+ B_{r}\right)\right\vert
    -\left(1-\left(1-\frac{s}{r}\right)^d\right)\left\vert B_{r}\right\vert.
    \]
\end{lem}
\begin{proof}
Since $z+B_{r-s}\subseteq z+B_{r}$, using the monotonicity of Lebesgue measure, we easily obtain
\[
    \left\vert B_{r-s}\right\vert - \left\vert A \cap \left(z+ B_{r-s}\right) \right\vert \leq 
    \left\vert B_{r}\right\vert - \left\vert A \cap \left( z+ B_{r}\right) \right\vert,
    \]
which, after a trivial calculation, implies the desired claim.
\end{proof}

\begin{prop}\label{prop: appendix main}
Let $A \subseteq \mathbbm{R}^d$ be a measurable set. We claim that
\[
\overline{\delta}(A) = \lim_{R \to \infty}\sup_{z \in \mathbbm{R}^d} \frac{\vert A \cap (z+[0,R]^d)\vert}{R^d}.
\]
\end{prop}
\begin{proof}
    Let $\varepsilon>0$ be arbitrary. 
    Notice that for an arbitrary $r>0$, we have
    \[
    \sup_{z \in \mathbbm{R}^d} \frac{\vert A \cap (z+[0,r]^d)\vert}{r^d}=\sup_{z \in \mathbbm{R}^d} \frac{\left \vert A \cap \left(z + B_{\frac{r}{2}}\right) \right \vert}{\left\vert B_{\frac{r}{2}} \right\vert},
    \]
    where $B_{r}$ now stands for the ball of radius $r$ around the origin, with respect to the $\ell^{\infty}$ norm on $\mathbbm{R}^d$.
    By definition of the limit superior, there exists some $R_0>0$ such that for all $r \geq R_0$
    \begin{equation}\label{eq: relative density upper estimate}
    \sup_{z \in \mathbbm{R}^d} \frac{\left \vert A \cap \left(z + B_{r}\right) \right \vert}{\left\vert B_{r} \right\vert} < \overline{\delta}(A) + \varepsilon.
    \end{equation}
    Let $S\geq 2R_0$ be arbitrary. We can now find $R> \frac{S}{2}$ large enough and $z \in \mathbbm{R}^d$ such that 
    \[\left(1-\left(1-\frac{S}{2R}\right)^d\right) < \frac{\varepsilon}{2}\] and 
    \[
    \frac{\left \vert A \cap \left(z + B_R\right) \right \vert}{\left\vert B_R \right\vert} > \overline{\delta}(A) - \frac{\varepsilon}{2}.
    \]
    We apply Lemmas \ref{lem: FKY 1} and \ref{lem: FKY 2} with $r=R$ and $s=\frac{S}{2}$ to conclude
    \begin{align*}
        \frac{1}{\left\vert B_R \right\vert} \int_{z + B_R} \frac{\left\vert A \cap\left(y+ B_\frac{S}{2}\right)\right\vert}{\left\vert B_\frac{S}{2} \right\vert} \textup{d}y
        &\geq \frac{\left\vert A \cap\left(z+ B_{R-\frac{S}{2}}\right)\right\vert}{\left\vert B_R \right\vert}\\
        &\geq \frac{\left\vert A \cap\left(z+ B_{R}\right)\right\vert}{\left\vert B_R \right\vert}
        -\left(1-\left(1-\frac{S}{2R}\right)^d\right) \\
        &> \overline{\delta}(A) - \frac{\varepsilon}{2} - \frac{\varepsilon}{2}
        = \overline{\delta}(A) - \varepsilon.
    \end{align*}
A standard measure-theoretic pigeonholing argument now ensures that there exists some $y \in \mathbbm{R}^d$ such that
\[
\frac{\left\vert A \cap\left(y+ B_{\frac{S}{2}}\right)\right\vert}{\left\vert B_{\frac{S}{2}} \right\vert}
> \overline{\delta}(A) - \varepsilon.
\]
This combined with \eqref{eq: relative density upper estimate} gives us
\[
\left \vert \sup_{z \in \mathbbm{R}^d} \frac{\left \vert A \cap \left(z + B_{\frac{S}{2}}\right) \right \vert}{\left\vert B_{\frac{S}{2}} \right\vert} -  \overline{\delta}(A)\right \vert < \varepsilon,
\]
for all $S \geq 2R_0$. Since $\varepsilon>0$ was arbitrary, this is precisely the claim of the proposition.
\end{proof}}

\section*{Statements and Declarations}
The work was supported by the \emph{Croatian Science Foundation} under the project number HRZZ-IP-2022-10-5116 (FANAP). {This paper was also supported in part by the European Union -- NextGenerationEU through the National Recovery and Resilience Plan 2021--2026; institutional grant of University of Zagreb Faculty of Science IK IA 1.1.3. Impact4Math.
}
\section*{Acknowledgements}

{The author is grateful to Professor Vjekoslav Kovač for his patient guidance, helpful comments and stimulating discussions and also to Shobu Shiraki for plenty of insightful discussions, comments and suggestions. The author further wishes to express gratitude to Professor Alex Iosevich for the suggestion that studying the connection between VC dimension and flexible configurations in Euclidean spaces might lead to interesting results. {Furthermore, the author wishes to thank two anonymous referees whose comments greatly improved the text.}}

\begingroup
\small
\bibliography{DensityRamsey}{}
\bibliographystyle{plainurl}
\endgroup
\end{document}